\documentclass[10pt, reqno]{amsart}

\usepackage{fullpage}
\usepackage{amsmath,amsthm,amsfonts,amssymb}
\usepackage{enumitem}

\usepackage{amsaddr}
\usepackage[colorlinks=true,citecolor=blue,linkcolor=red]{hyperref}

\makeatletter
\def\ps@runheadfooter{%
  \ps@empty
  \def\@evenhead{%
    \setTrue{runhead}%
    \normalfont\scriptsize
    \hfil
    \def\thanks{\protect\thanks@warning}%
    \leftmark{}{}\hfil}%
  \def\@oddhead{%
    \setTrue{runhead}%
    \normalfont\scriptsize
    \hfil
    \def\thanks{\protect\thanks@warning}%
    \rightmark{}{}\hfil}%
  \def\@oddfoot{\normalfont\scriptsize\hfil\thepage\hfil}%
  \let\@evenfoot\@oddfoot
  \let\@mkboth\markboth
}
\makeatother

\newtheorem{thm}{Theorem}[section]
\newtheorem{lem}[thm]{Lemma}
\newtheorem{cor}[thm]{Corollary}
\newtheorem{prop}[thm]{Proposition}
\theoremstyle{definition}
\newtheorem{defn}[thm]{Definition}
\newtheorem{exmp}[thm]{Example}

\newtheorem{conv}[thm]{Convention}
\newtheorem{prob}[thm]{Open problem}
\theoremstyle{remark}
\newtheorem{rem}[thm]{Remark}

\renewcommand{\phi}{\varphi}
\newcommand{\CA}{\mathcal A}
\newcommand{\CB}{\mathcal B}
\newcommand{\CC}{\mathcal C}
\newcommand{\CG}{\mathcal G}
\newcommand{\CH}{\mathcal H}
\newcommand{\CP}{\mathcal P}
\newcommand{\rk}{\operatorname{rk}}
\newcommand{\val}{\operatorname{val}}
\newcommand{\red}{\operatorname{red}}
\newcommand{\cycred}{\operatorname{cycred}}
\newcommand{\xsupp}{\operatorname{supp}}
\newcommand{\blen}{\operatorname{bl}}
\newcommand{\enc}{\operatorname{enc}}
\newcommand{\CPrim}{\mathsf{CPrim}}
\newcommand{\SLP}{\mathsf{SLP}}
\newcommand{\immq}{\mathrel{\twoheadrightarrow_{\rm imm}}}
\newcommand{\ff}{\mathrel{\leq_{\!*}}}
\newcommand{\eps}{\varepsilon}

\newcommand{\Aut}{\operatorname{Aut}}
\newcommand{\Out}{\operatorname{Out}}

\newcommand{\Z}{\mathbb Z}

\begin{document}

\title[Compressed primitivity problem in free groups]{Compressed primitivity problem in free groups}

\author[Ilya Kapovich]{Ilya Kapovich}

\address{Department of Mathematics and Statistics, Hunter College of CUNY\newline
  \indent 695 Park Ave, New York, NY 10065, U.S.A.
  \newline \indent e-mail \texttt{ik535@hunter.cuny.edu}
  \newline\indent ORCID 0000-0002-7694-6236
  }

\keywords{free group, primitive element, straight-line program, compressed word, Whitehead automorphism, Whitehead minimization, Stallings folding, computational complexity}

\hypersetup{
  bookmarksdepth=2,
  pdftitle={Compressed primitivity problem in free groups},
  pdfauthor={Ilya Kapovich},
  pdfkeywords={free group, primitive element, straight-line program, compressed word, Whitehead automorphism, Whitehead minimization, Stallings folding, computational complexity}
}

\subjclass[2020]{Primary 20F10, Secondary 20E05, 68Q25, 68Q45}

\date{}

\begin{abstract}
For a fixed integer $r\ge 2$, we prove that the \emph{compressed primitivity problem} in the free group $F_r=F(x_1,\dots,x_r)$ is decidable in nondeterministic polynomial time. That is, for a  \emph{straight-line program}
$\CA$ over $\{x_1,\dots,x_r\}^{\pm1}$ representing an element $g\in F_r$,  the problem of deciding
whether $g$ is primitive in $F_r$ belongs to $\mathsf{NP}$, with input measured
by the size of $\CA$. For $r=2$, we prove that this problem is decidable in
deterministic polynomial time. We also show that, in every fixed rank
$r\ge 2$, automorphic minimality of the conjugacy class of a compressed word
in $F_r$ is decidable in deterministic polynomial time.
\end{abstract}

\maketitle
\pagestyle{runheadfooter}
\enlargethispage{3pt}

\setcounter{tocdepth}{1}
\tableofcontents

\section{Introduction}\label{s:introduction}

Straight-line programs, a data-compression tool from computer science, are
useful for algorithmic problems in group theory whose inputs may grow
exponentially under algebraic operations. The word problem for $\Aut(F_r)$,
where $F_r=F(x_1,\dots,x_r)$ has finite rank $r\ge2$, is typical: a product
$\phi_n$ of $n$ generators may send a basis letter $x_i$ to a word of
exponential length, even when that word freely reduces to $x_i$. Naive
successive substitution may therefore require exponential time and space.

Schleimer~\cite{Schleimer} instead represented $\phi_n(x_i)$ by a
polynomial-size \emph{compressed word} over
$\Sigma_r=\{x_1,\dots,x_r\}^{\pm1}$ and performed free reduction directly in
compressed form. This gave deterministic polynomial-time algorithms for the
word problems in $\Aut(F_r)$, $\Out(F_r)$, and $F_r\rtimes_\phi\Z$ with
$\phi\in\Aut(F_r)$. Briefly, a \emph{straight-line program} ($\SLP$) is a
finite list of substitution rules, each of which uses only letters of the
output alphabet and rules already listed. It produces a single word by
successive substitution. Its size is the number of rules plus the total
number of symbols on their right-hand sides. An $\SLP$ of size $n$ can
produce a word of exponential length. Section~\ref{s:slp} gives the precise
definition and fixes an encoding whose length differs from this grammar size
by at most a logarithmic factor. All complexity bounds below are measured in
the compressed input size, not the length of the expanded word.

Schleimer's work initiated an active study of compressed
group-theoretic problems; see Lohrey's survey~\cite{LohreySurvey} and the
account of Lohrey--Schleimer~\cite{LohreySchleimer}. Polynomial-time
algorithms are known for compressed word, conjugacy, simultaneous-conjugacy,
and centralizer problems in hyperbolic groups~\cite{HoltLohreySchleimer};
the same work proves $\mathsf{NP}$-completeness of compressed knapsack in
every infinite hyperbolic group. Polynomial-time algorithms are also known
for the compressed word and conjugacy problems in groups hyperbolic relative
to free abelian subgroups~\cite{HoltReesWord,HoltReesConjugacy}, and for
several compressed problems in finitely generated nilpotent
groups~\cite{MacdonaldMyasnikovNikolaevVassileva}. For free groups, Linton's
polynomial-time fully compressed subgroup-membership algorithm~\cite{Linton},
with a fixed number of compressed generators, is a key ingredient in our
$\mathsf{NP}$ verifier.

Strong hardness results are also known. Bartholdi, Figelius, Lohrey, and
Wei\ss~proved that the compressed word problem is $\mathsf{PSPACE}$-complete
for several natural groups, including Thompson's groups, the Grigorchuk and
Gupta--Sidki groups, and wreath products $G\wr\Z$ with $G$ finite
non-solvable or free of rank at least two~\cite{BartholdiFigeliusLohreyWeiss}.
W\"achter and Wei\ss~constructed an automaton group whose compressed word
problem is $\mathsf{EXPSPACE}$-complete~\cite{WachterWeiss}.

We study compressed primitivity in $F_r=F(X_r)$, where
$X_r=\{x_1,\dots,x_r\}$ and $\Sigma_r=X_r\sqcup X_r^{-1}$. An element is
\emph{primitive} if it belongs to a free basis. For $g\ne1$, this is equivalent
to $\langle g\rangle$ being a rank-one free factor, or to $g=\phi(x_1)$ for
some $\phi\in\Aut(F_r)$. For an $\SLP$ $\CA$ over $\Sigma_r$, let $w_{\CA}$
be its output and $g_{\CA}\in F_r$ the represented element. The
\emph{compressed primitivity problem} asks whether $g_{\CA}$ is primitive.
We denote by $\CPrim_r$ the language of encodings of such $\SLP$s with
primitive output; the formal encoding and language are defined in
Section~\ref{ss:encoding}.

The class $\mathsf P$ consists of decision problems solvable in deterministic
polynomial time. A decision problem belongs to $\mathsf{NP}$ if each positive
instance admits a certificate of polynomial bit length whose validity can be
checked in deterministic polynomial time, and no negative instance has a
valid certificate. This is equivalent to decidability in nondeterministic
polynomial time. The class $\mathsf{coNP}$ consists of languages whose
complements belong to $\mathsf{NP}$. The distinction between finding a
certificate and verifying one is central here.

Our main result, restated as Theorem~\ref{t:main1}, is:

\begin{thm}\label{t:main}
For every fixed integer $r\ge 2$, the language $\CPrim_r$ belongs to
$\mathsf{NP}$. Equivalently, the compressed primitivity problem in $F_r$ is
decidable in nondeterministic polynomial time.
\end{thm}

For $F_2=F(a,b)$, Theorem~\ref{t:rank-two-main} gives the stronger conclusion:

\begin{thm}\label{t:rank-two-intro}
The compressed primitivity problem for $F_2=F(a,b)$ is decidable in
deterministic polynomial time. That is, $\CPrim_2\in\mathsf P$.
\end{thm}

\smallskip
\noindent\emph{The certificate and the proof strategy.}
The identity is rejected. After compressed cyclic reduction of a nontrivial
input, consider the cycle $S_w$ with one labeled edge for each letter of the
resulting nonempty word $w$. Let $s$ be the number of
basis generators occurring in $w$, with either sign. The free-factor
criterion of Silva--Weil~\cite[Corollary~3.4]{SilvaWeil}, also expressed by
Puder's quotient-distance theorem~\cite{Puder,PuderParzanchevski}, says in
this setting that primitivity is equivalent to the existence of $s-1$
vertex identifications followed by Stallings folds. A vertex identification
is also called a \emph{pinch}; folding identifies equally labeled outgoing
edges. The final graph is the rose with one vertex and one loop for each of
the $s$ generators. Related criteria for primitive sets appear in
Clifford--Goldstein~\cite{CliffordGoldstein}.

Theorem~\ref{t:prefix-certificate} puts this existing criterion in a form
suited to compressed verification. Every identified vertex can be specified
by a position in the \emph{original} cycle. Thus the certificate consists of
$s-1$ pairs of binary positions $(p_i,q_i)$. If $a_t$ denotes the prefix of
$w$ of length $t$, these pairs determine the elements
$a_{p_i}a_{q_i}^{-1}$, called \emph{prefix differences}. The certificate is
valid precisely when $w$ and these differences generate the free group on
the generators occurring in $w$. Prefixes and their differences have
polynomial-size $\SLP$s, and Linton's compressed subgroup-membership
algorithm~\cite{Linton} verifies this final subgroup equality in polynomial
time for fixed $r$. Neither the exponentially large cycle nor the
intermediate folded graphs need be constructed.

\smallskip
\noindent\emph{Rank two and Whitehead minimality.}
In rank two, the Nielsen--Osborne--Zieschang classification
\cite{OsborneZieschang} says that a primitive conjugacy class is determined
by its coprime exponent-sum pair. We apply the Euclidean algorithm to that
pair and, after inverting or interchanging the basis elements as needed,
realize each division by an automorphism of the form
\[
 a\mapsto ab^{-m},\qquad b\mapsto b\qquad(m\ge1).
\]
The integer $m$ is stored in binary and the power of $b$ is represented by an
$\SLP$, rather than expanding the automorphism into $m$ individual moves.
A final compressed cyclic-length test then decides primitivity. Section~\ref{s:rank-two}
separates the classical algebraic criterion, the arithmetic, and its
compressed implementation.

We also prove that, for fixed $r$, \emph{automorphic minimality} of a
compressed conjugacy class is decidable in deterministic polynomial time
(Theorem~\ref{t:compressed-whitehead-minimality}). Here minimality means
that no automorphism decreases the length of a cyclically reduced
representative. Whitehead's classical theorem reduces this to checking a
finite family of automorphisms. Their length changes are determined by the
\emph{weighted Whitehead graph}, which records cyclic two-letter counts;
Section~\ref{ss:whitehead-background} recalls the definitions and formula.
We compute these counts directly from the grammar.

A minimality test should not be confused with an algorithm that finds a
minimal representative. Classical Whitehead descent repeatedly applies a
length-reducing Whitehead automorphism and cyclically reduces. Although a
prescribed sequence of moves can be evaluated efficiently in compressed
form, the number of decreasing moves need not be polynomial in the
compressed input size. Example~\ref{ex:slow-whitehead-descent} makes this
obstruction explicit. The rank-two primitivity algorithm avoids it by
batching powers of elementary moves.

\smallskip
\noindent\emph{Related language-theoretic results.}
Ciobanu--Elder--Ferov~\cite[Proposition~9 and Corollary~15]{CiobanuElderFerov}
show that primitive elements and bases of $F_2$, represented by ordinary
reduced words, form EDT0L languages, a class defined using controlled
parallel rewriting. In arbitrary fixed finite rank they obtain
context-sensitive languages of primitive elements and primitive sets. Their
recognition procedure uses nondeterministic linear space on uncompressed
input~\cite[Proposition~14]{CiobanuElderFerov}. These results concern the
explicit words, whereas the present bounds use the size of an $\SLP$ whose
expansion may be exponentially longer. The better behavior in rank two
occurs in both settings.

The results of this paper naturally raise the following questions:

\begin{prob}\label{p:CPrim_r}
Let $r\ge 3$ be a fixed integer. Does $\CPrim_r$ belong to $\mathsf P$? As an a priori easier question, does $\CPrim_r$ belong to $\mathsf{coNP}$?
\end{prob}
Theorem~\ref{t:main} gives $\CPrim_r\in\mathsf{NP}$. Membership in
$\mathsf{coNP}$ would strengthen this, but is weaker than membership in
$\mathsf P$.

Roig, Ventura, and Weil~\cite{RoigVenturaWeil} solve Whitehead
minimization for uncompressed words in polynomial time, even when $r$ is part
of the input. For fixed $r$, the compressed version asks, given an $\SLP$
$\CA$, for an $\SLP$ $\CB$ producing a freely and cyclically reduced word
minimal in the automorphic orbit of $w_{\CA}$. Decompression followed by
Whitehead minimization gives a deterministic exponential-time solution; its
optimality is unknown:

\begin{prob}\label{p:compressed-minimization}
Let $r\ge 2$ be a fixed integer. What can be said about the complexity of the Whitehead minimization problem for compressed words in $F_r$?
\end{prob}

The algorithms proved here do not settle Problem~\ref{p:compressed-minimization}, even for $F_2$. Our deterministic polynomial time primitivity algorithm in Theorem~\ref{t:rank-two-intro} uses combinatorics
specific to primitive elements in $F_2$ and does not extend to other
$\Aut(F_2)$-orbits.

\smallskip
\noindent\emph{Organization.}
Section~\ref{s:slp} develops the compressed-word conventions and basic
algorithms. Section~\ref{s:core} recalls Stallings graphs and the free-factor
criteria, including their relation to the earlier literature.
Section~\ref{s:certificate} gives the positional prefix-difference
certificate, and Section~\ref{s:verification} proves its polynomial-time
verifiability and hence Theorem~\ref{t:main}; its final subsection discusses
endpoint verification and the dependence on the fixed rank.
Section~\ref{s:whitehead-minimization} proves the independent
Whitehead-minimality result and explains why efficient evaluation of
prescribed automorphism sequences does not imply efficient minimization.
Section~\ref{s:rank-two} proves Theorem~\ref{t:rank-two-intro} and records its
quantitative complexity bound.

The author thanks the referees for helpful suggestions.

\section{Compressed words and basic algorithms}\label{s:slp}

Unless stated otherwise, the rank $r$ and the terminal alphabet are fixed
and are not part of the input. Implicit constants in $O$-bounds may depend
on this fixed rank and on maps or finite families explicitly fixed in a
statement, but not on the input $\SLP$s or the length of an input sequence.
Dependence on the number of subgroup generators is displayed explicitly
when Linton's algorithm is used.

\subsection{Words and straight-line programs}\label{ss:words-slps}
For an alphabet $Y$, the free monoid $Y^*$ consists of finite strings of
letters from $Y$, with concatenation as multiplication. For $w\in Y^*$,
write $|w|$ for its number of letters. The letters of
$\Sigma_r=X_r\sqcup X_r^{-1}$ are called \emph{terminals}; an $\SLP$ also
uses auxiliary symbols, called \emph{nonterminals}, for previously defined
strings. A \emph{terminal word} is a string containing only terminals.
Equality of \emph{literal words} means equality of strings in the free
monoid, before imposing cancellation in $F_r$.

For a word $u=z_1\cdots z_N\in\Sigma_r^*$, put
\[
 u^{-1}=z_N^{-1}\cdots z_1^{-1}.
\]
The natural monoid map $\Sigma_r^*\to F_r$ sends a terminal word to the group element that it represents. A word is \emph{freely reduced} if it has no factor $zz^{-1}$ with $z\in\Sigma_r$. Every word $u$ has a unique freely reduced form, denoted $\red(u)$, representing the same element of $F_r$. A freely reduced word is \emph{cyclically reduced} if it is empty or if its first and last letters are not mutually inverse. Starting with $\red(u)$ and repeatedly deleting mutually inverse first and last letters gives a uniquely determined factor $\cycred(u)$ and a freely reduced word $c$ such that, as a literal word,
\[
 \red(u)=c\,\cycred(u)\,c^{-1}.
\]
We call $\cycred(u)$ the cyclically reduced form of $u$ and write
\[
 \lVert u\rVert_{\rm cyc}=|\cycred(u)|
\]
for the cyclic length of the represented conjugacy class. The empty word is denoted by $\eps$.

\begin{defn}[straight-line programs and their size]\label{d:slp}
A \emph{straight-line program}, abbreviated $\SLP$, over $\Sigma_r$ is a tuple
\[
 \CA=(\Sigma_r,\mathcal V,A_m,\mathcal P),
 \qquad \mathcal V=\{A_1,\dots,A_m\},
\]
where $m\ge1$, $A_m$ is the root nonterminal, and
\[
 \mathcal P=\{A_i\longrightarrow W_i:1\le i\le m\}
\]
contains exactly one production for each nonterminal, with
\[
 W_i\in\bigl(\Sigma_r\sqcup\{A_1,\dots,A_{i-1}\}\bigr)^*.
\]
The empty right-hand side is allowed. Every nonterminal $A_j$ occurring
on the right-hand side of the rule for $A_i$ must satisfy $j<i$. Thus a
rule refers only to earlier rules, and the dependency graph is acyclic.

For each nonterminal, define its terminal expansion recursively by substituting the already defined expansions of the earlier nonterminals. We write $w_{A_i}\in\Sigma_r^*$ for the word produced by $A_i$ and
\[
 w_{\CA}=w_{A_m}
\]
for the uncompressed terminal word produced by the program. Equivalently, $w_{\CA}=\val(\CA)$ in the notation used elsewhere in the paper.

The \emph{size} of $\CA$ is the structural grammar size
\begin{equation}\label{eq:slp-size}
 \boxed{\displaystyle
 \lVert\CA\rVert
 =m+\sum_{i=1}^m |W_i|.}
\end{equation}
Thus every nonterminal and right-hand-side symbol has unit cost.
\end{defn}

\subsection{Encoding and size}\label{ss:encoding}
For a nonnegative integer $t$, let
\[
 \blen(t)=\max\bigl\{1,\lceil\log_2(t+1)\rceil\bigr\},
\]
the number of bits in the ordinary binary representation, with $0$ assigned
one bit.

\begin{defn}[canonical string encoding and associated languages]\label{d:slp-encoding}
Let $\operatorname{bin}(t)\in\{\mathtt{0},\mathtt{1}\}^*$ be the ordinary binary expansion of a nonnegative integer $t$, with no leading zeroes and with $\operatorname{bin}(0)=\mathtt{0}$. Fix the finite encoding alphabet
\[
 \Omega_r=
 \Sigma_r\sqcup
 \{\mathtt{0},\mathtt{1},\mathtt{N},\mathtt{L},\mathtt{R},
   \mathtt{B},\mathtt{P},\mathtt{Q},\mathtt{E}\},
\]
where the displayed meta-symbols are all distinct from one another and from the terminal letters in $\Sigma_r$. The letters $\mathtt{L},\mathtt{R}$ delimit the whole program, $\mathtt{B}$ terminates the initial binary encoding of the number of nonterminals, $\mathtt{P}$ separates the left- and right-hand sides of a production, $\mathtt{Q}$ terminates a nonterminal reference on a right-hand side, and $\mathtt{E}$ terminates a production.

For symbols occurring on the right-hand side of an $\SLP$ production,
define
\[
 \operatorname{code}(z)=z\quad(z\in\Sigma_r),
 \qquad
 \operatorname{code}(A_j)
 =\mathtt{N}\,\operatorname{bin}(j)\,\mathtt{Q}.
\]
If $W=C_1\cdots C_t$, let
$\operatorname{code}(W)=\operatorname{code}(C_1)\cdots
\operatorname{code}(C_t)$, with the empty word used when $t=0$. Encode the production $A_i\to W_i$ by
\[
 \operatorname{pcode}(A_i\to W_i)
 =\mathtt{N}\,\operatorname{bin}(i)\,\mathtt{P}\,
   \operatorname{code}(W_i)\,\mathtt{E},
\]
and encode the whole program by the word
\begin{equation}\label{eq:slp-string-code}
 \enc_r(\CA)=
 \mathtt{L}\,\operatorname{bin}(m)\,\mathtt{B}\,
 \operatorname{pcode}(A_1\to W_1)\cdots
 \operatorname{pcode}(A_m\to W_m)\,\mathtt{R}
 \in\Omega_r^*.
\end{equation}
Put
\[
 \beta(\CA)=|\enc_r(\CA)|.
\]
The delimiters make this code uniquely decodable. In particular, from a string one can parse $m$, the productions in their prescribed order, and every terminal or nonterminal occurrence, and then check the conditions $j<i$ on the right-hand side of the $i$th production. Thus validity of an $\SLP$ code can be checked in deterministic time polynomial in $\beta(\CA)$.

The set
\[
 \SLP_r=
 \{\enc_r(\CA):\CA\text{ is an $\SLP$ over }\Sigma_r\}
 \subseteq\Omega_r^*
\]
is a formal language, and every collection of $\SLP$s over $\Sigma_r$ is
identified with a sublanguage of $\SLP_r$. In particular, define
\[
 \CPrim_r=
 \{\enc_r(\CA):\CA\text{ is an $\SLP$ over }\Sigma_r
 \text{ and }g_{\CA}\text{ is primitive in }F_r\}
 \subseteq\Omega_r^*.
\]
Invalid encodings do not belong to this language and are rejected.
\end{defn}

\begin{prop}[structural size and encoded length]
\label{p:slp-size-comparison}
For every $\SLP$ $\CA$ over $\Sigma_r$,
\begin{equation}\label{eq:structural-encoding-comparison}
 \lVert\CA\rVert
 \le \beta(\CA)
 \le
 (\lVert\CA\rVert+1)
 \bigl(\blen(\lVert\CA\rVert)+3\bigr).
\end{equation}
In particular,
\[
 \beta(\CA)
 =O\!\left(\lVert\CA\rVert
       \log(\lVert\CA\rVert+2)\right),
\]
with a constant independent of $\CA$.
\end{prop}

\begin{proof}
Put $s=\lVert\CA\rVert$ and let $m$ be the number of nonterminals.  Every
production contributes at least one symbol to its code, and the code of every
right-hand-side symbol has positive length.  Since $s$ is the number of
nonterminals plus the number of right-hand-side symbols, this gives
$s\le\beta(\CA)$.

For the upper bound, every nonterminal index is at most $m\le s$ and hence
has at most $\blen(s)$ bits.  The three outer delimiters together with the
initial binary encoding of $m$ contribute at most $3+\blen(s)$.  Each of the
$m$ productions contributes at most $\blen(s)+3$ symbols apart from its
right-hand side, and the code of each right-hand-side symbol has length at
most $\blen(s)+2$.  Since there are $s-m$ right-hand-side symbols,
\[
 \begin{aligned}
 \beta(\CA)
 &\le 3+\blen(s)+m(\blen(s)+3)
       +(s-m)(\blen(s)+2)\\
 &\le (s+1)(\blen(s)+3).
 \end{aligned}
\]
This proves~\eqref{eq:structural-encoding-comparison}.  The asymptotic
statement follows from $\blen(s)=O(\log(s+2))$.
\end{proof}

All algorithms below take the canonical code as input. By
Proposition~\ref{p:slp-size-comparison}, polynomial running-time, output-size,
and certificate-size bounds may use either size measure. Explicit
bit-operation bounds refer to encoded length and may retain its logarithmic
factor. Since $\Omega_r$ is fixed, conversion to a binary alphabet changes
length only by a constant factor.

\subsection{Basic compressed-word operations}\label{ss:slp-operations}
An $\SLP$ is in \emph{Chomsky normal form} if every production has one of the forms
\[
 A_i\longrightarrow z\quad(z\in\Sigma_r),
 \qquad
 A_i\longrightarrow A_jA_k\quad(j,k<i),
\]
with the additional exceptional rule $A_m\to\eps$ permitted when the
output word is empty. First replace every terminal occurring in a right-hand
side of length greater than one by a shared terminal nonterminal. Since
$\Sigma_r$ is fixed, this adds only $O(1)$ new productions and
right-hand-side symbols, while leaving the lengths of the original right-hand
sides unchanged. Long right-hand sides can then be binarized by introducing
binary trees of dummy nonterminals. Empty references can be deleted recursively, and unit
productions can be eliminated by redirecting references in the acyclic
dependency order, without copying terminal expansions. If the root expansion
is empty, replace the program by the one-rule program $A_1\to\eps$; otherwise,
discard all empty and root-unreachable nonterminals. Relabel the remaining
variables in dependency order. Thus the total number of nonterminals and
right-hand-side symbol occurrences increases by at most a constant
multiplicative factor. Consequently every $\SLP$ of size
$n$ can be converted in polynomial time to an equivalent $\SLP$ in Chomsky
normal form of size $O(n)$.  This is the normal form used by
Schleimer~\cite{Schleimer}.  An $\SLP$ with $m$
nonterminals in this form has size $O(m)$ and encoded length
$O(m\log(m+2))$ by Proposition~\ref{p:slp-size-comparison}.

For example, the productions
\[
 A_1\to x_1,
 \qquad A_i\to A_{i-1}A_{i-1}\quad(2\le i\le m)
\]
produce $w_{A_m}=x_1^{2^{m-1}}$ with size $O(m)$ and encoded length
$O(m\log(m+2))$.

Letters in $u=u_1\cdots u_N$ are numbered $1,\dots,N$, while the
\emph{cut positions}, or equivalently prefix lengths, are numbered
$0,\dots,N$. For $0\le i\le j\le N$, write
\[
 u[i:j]=u_{i+1}\cdots u_j.
\]
Thus $u[i:j]$ has length $j-i$, $u[0:t]$ is the prefix of length $t$, and
$u[0:0]=\eps$.

A \emph{composition system} is defined like an $\SLP$, except that a
right-hand side may contain a truncated nonterminal $B[i:j]$, representing
the factor $w_B[i:j]$ between cut positions $i$ and $j$. These indices
satisfy $0\le i\le j\le |w_B|$ and are written in binary; their bit lengths
are included in the input size. Hagenah's conversion algorithm, presented
in~\cite[Theorem~7.1]{Schleimer}, converts a composition system to an
equivalent $\SLP$ in polynomial time.

For a terminal word $u$, its \emph{exponent sum} $\sigma_{x_j}(u)$ is the
number of occurrences of $x_j$ minus the number of occurrences of
$x_j^{-1}$. The vector of exponent sums is its image in the abelianization
$F_r/[F_r,F_r]\cong\mathbb Z^r$. It is obtained from the Parikh vector,
which counts the occurrences of all $2r$ terminal letters, by subtracting
the coordinate for each inverse letter from that for the corresponding
positive letter.

\begin{lem}[standard compressed-word operations]\label{l:slp-ops}
Let $\CA$ be an $\SLP$ of size $n$ over the fixed alphabet $\Sigma_r$.
\begin{enumerate}
\item One can convert $\CA$ to Chomsky normal form in polynomial time. The lengths $|w_{A_i}|$ of all nonterminal expansions, and in particular $|w_{\CA}|$, can be computed as binary integers in polynomial time. Moreover,
\[
 |w_{\CA}|\le 2^{O(n)},
\]
so every such length has $O(n)$ bits.
\item Given binary integers $i,j$ with $0\le i\le j\le |w_{\CA}|$, one can construct, in time polynomial in $n+\blen(i)+\blen(j)$, an $\SLP$ producing the factor $w_{\CA}[i:j]$.
\item Given $\SLP$s for words $u$ and $v$, one can construct in polynomial time $\SLP$s for $uv$ and $u^{-1}$. In particular, one can construct an $\SLP$ for $w_{\CA}^{-1}$.
\item The support
\[
 \xsupp_{X_r}(w_{\CA})
 =\{j:x_j^{\pm1}\text{ occurs in }w_{\CA}\}
\]
and the exponent-sum vector
\[
 \bigl(\sigma_{x_1}(w_{\CA}),\dots,
       \sigma_{x_r}(w_{\CA})\bigr)\in\mathbb Z^r
\]
can be computed in polynomial time.
\end{enumerate}
\end{lem}

\begin{proof}
For (a), use the normalization described above. Since the number of
nonterminals and right-hand-side symbols in the input is $O(n)$, the
normalized program has size $O(n)$ and encoded length
$O(n\log(n+2))$ by Proposition~\ref{p:slp-size-comparison}. In Chomsky
normal form, a nonterminal of height $h$ has output length at most $2^h$,
while the height is $O(n)$. Hence $|w_{\CA}|\le 2^{O(n)}$. Applying the same
argument to the subprogram rooted at each original nonterminal $A_i$ gives
$|w_{A_i}|\le 2^{O(n)}$ uniformly in $i$.

The lengths of the original nonterminal expansions are then computed directly
in dependency order. Each terminal occurrence contributes one, each occurrence
of $A_j$ contributes the already computed integer $|w_{A_j}|$, and an empty
right-hand side contributes zero. Since there are $O(n)$ right-hand-side
symbols and every integer involved has $O(n)$ bits, the computation takes
polynomial time. Compare Lemmas~2.5 and~2.6 of~\cite{Schleimer}.

For (b), adjoin to $\CA$ a new root production
\[
 P_{i,j}\longrightarrow A_m[i:j].
\]
This is a composition system whose encoding length is polynomial in
$n+\blen(i)+\blen(j)$. Apply the composition-system-to-$\SLP$ conversion of Theorem~7.1 in~\cite{Schleimer}.

For (c), first normalize the input programs as in (a) and rename their
nonterminals so that the variable sets are disjoint. For concatenation, take
their disjoint union and add a new root whose production is the product of the
two old roots. For inversion, take a disjoint inverse copy of the
nonterminals. If $A_i\to A_jA_k$, put
$A_i^{-1}\to A_k^{-1}A_j^{-1}$, and if $A_i\to z\in\Sigma_r$, put
$A_i^{-1}\to z^{-1}$; the empty word is fixed. The resulting programs have size polynomial in the total input size; their
encoded lengths are polynomial as well by Proposition~\ref{p:slp-size-comparison}.

For (d), normalize as in (a), and associate to every nonterminal its support
subset and exponent-sum vector. At a terminal production these data are
immediate; at an empty production they are the empty set and the zero
vector. At a concatenation production one takes the union of supports
and adds vectors. Since $r$ is fixed, the supports have constant size, while
every exponent sum has $O(n)$ bits by (a).
\end{proof}

\subsection{Compressed free and cyclic reduction}\label{ss:reduction}
The following form of Schleimer's compressed-reduction result is the
preprocessing theorem used throughout the paper.

\begin{prop}[Schleimer: compressed free and cyclic reduction]\label{p:cyclic-reduction}
Let $\CA$ be an $\SLP$ over $\Sigma_r$.
\begin{enumerate}
\item There is a deterministic polynomial-time algorithm that constructs an $\SLP$ $\CB_{\rm fr}$ satisfying
\[
 w_{\CB_{\rm fr}}=\red(w_{\CA}).
\]
The size $\lVert\CB_{\rm fr}\rVert$ is bounded by a polynomial in $\lVert\CA\rVert$, and the integer
$|\red(w_{\CA})|=|w_{\CB_{\rm fr}}|$ can be computed in binary in polynomial time.
\item There is a deterministic polynomial-time algorithm that constructs $\SLP$s $\CB_{\rm cyc}$ and $\CC$ satisfying
\[
 w_{\CB_{\rm cyc}}=\cycred(w_{\CA})
\]
and
\[
 \red(w_{\CA})
 =w_{\CC}\,w_{\CB_{\rm cyc}}\,w_{\CC}^{-1}
\]
as terminal words. In particular, $w_{\CB_{\rm cyc}}$ is freely and cyclically reduced and represents a conjugate of $g_{\CA}$. The sizes of both output programs are bounded by a polynomial in $\lVert\CA\rVert$, and the cyclic length
$\lVert w_{\CA}\rVert_{\rm cyc}=|w_{\CB_{\rm cyc}}|$ can be computed in binary in polynomial time.
\end{enumerate}
\end{prop}

\begin{proof}
Put the input in Chomsky normal form. Theorem~3.3 of~\cite{Schleimer} constructs, in polynomial time, a composition system producing the free reduction of $w_{\CA}$. Hagenah's conversion theorem, Theorem~7.1 of~\cite{Schleimer}, converts it to the $\SLP$ $\CB_{\rm fr}$ with polynomial overhead. This proves (a), including the polynomial size bound; the length is then computed by Lemma~\ref{l:slp-ops}(a).

Starting with the freely reduced compressed word,
\cite[Corollary~3.6]{Schleimer} constructs a compressed cyclic reduction
together with a compressed conjugating word. A second application
of~\cite[Theorem~7.1]{Schleimer} converts the resulting composition systems
to the $\SLP$s $\CB_{\rm cyc}$ and $\CC$. The description
in~\cite[Corollary~3.6]{Schleimer} removes the maximal mutually inverse
prefix and suffix, so its output is precisely $\cycred(w_{\CA})$ as defined
above. The polynomial size and length claims follow as in (a).
\end{proof}

In particular, the compressed word problem in $F_r$ is decidable in polynomial time: $g_{\CA}=1$ if and only if $|\red(w_{\CA})|=0$. Since conjugation is an automorphism of $F_r$, $g_{\CA}$ is primitive if and only if the cyclically reduced element represented by $\CB_{\rm cyc}$ is primitive. The empty output represents the identity and is not primitive.

\section{Stallings graphs and free-factor criteria}\label{s:core}

This section recalls the uncompressed graph-theoretic facts used to prove
the certificate theorem. No exponentially large graph will be constructed
by the verification algorithm in Section~\ref{s:verification}.

\subsection{Stallings graphs}\label{ss:core-graphs}
Let $X$ be a finite basis, $\Sigma_X=X\sqcup X^{-1}$, and $F(X)$ the free
group on $X$. We recall Stallings' finite-graph construction; see
\cite{Stallings,KapovichMyasnikov} for background. The $X$-rose $R_X$ has
one vertex and one geometric loop for each $x\in X$, labeled $x$ in one
orientation and $x^{-1}$ in the other.

\begin{defn}[$X$-graphs, folding, and pointed cores]\label{d:x-graphs}
A finite \emph{$X$-graph} $\Gamma$ consists of a finite vertex set $V\Gamma$, a finite set $E\Gamma$ of oriented edges, a fixed-point-free involution $e\mapsto\bar e$ on $E\Gamma$, initial and terminal maps $o,t:E\Gamma\to V\Gamma$, and a label map
\[
 \mu:E\Gamma\longrightarrow\Sigma_X,
\]
such that
\[
 o(\bar e)=t(e),\qquad t(\bar e)=o(e),
 \qquad \mu(\bar e)=\mu(e)^{-1}.
\]
The pair $\{e,\bar e\}$ is one geometric edge. Loops and multiple geometric edges are allowed. An edge is \emph{outgoing} at $v$ when $o(e)=v$, and the degree of $v$ is the number of oriented edges $e$ with $o(e)=v$.
For a loop, both $e$ and $\bar e$ are outgoing at $v$, so the loop contributes
two to the degree.

An edge path $p=e_1\cdots e_n$ has label
$\mu(p)=\mu(e_1)\cdots\mu(e_n)\in\Sigma_X^*$ and is \emph{reduced} if $e_{i+1}\ne\bar e_i$ for every $i$. A pointed $X$-graph is a pair $(\Gamma,*)$ with a distinguished base vertex. A morphism of pointed $X$-graphs preserves the basepoint, incidence, the edge involution, and labels.

The graph $\Gamma$ is \emph{folded} if, for every vertex $v$ and every $z\in\Sigma_X$, there is at most one outgoing edge at $v$ labeled $z$. Equivalently, the canonical label map $\Gamma\to R_X$ to the $X$-rose is an immersion. A \emph{Stallings fold} identifies two distinct outgoing edges with the same label, together with their terminal vertices and inverse edges. Repeated folding terminates and produces a folded quotient, unique up to label-preserving isomorphism.

For a connected pointed $X$-graph $(\Gamma,*)$, its \emph{pointed core} is the union of the basepoint and all edges traversed by reduced closed paths based at $*$. A finite connected pointed folded $X$-graph is called a \emph{pointed core graph} when it equals its pointed core. Equivalently, every geometric edge occurs in some reduced based loop. Such a graph has no vertex of degree one except possibly the basepoint. The subgroup represented by $(\Gamma,*)$ is
\[
 H(\Gamma,*)=
 \{\overline{\mu(p)}\in F(X):p\text{ is a closed path based at }*\},
\]
where the bar denotes the group element represented by the path label. For a folded graph the label map induces an injection of the fundamental
group $\pi_1(\Gamma,*)$ into $F(X)$ with image $H(\Gamma,*)$.
\end{defn}

For every finitely generated subgroup $H\le F(X)$ there is, up to pointed label-preserving isomorphism, a unique finite pointed core graph representing $H$; it is denoted $\Gamma_X(H)$. One may construct it from a bouquet of based loops spelling a finite generating set for $H$, then perform Stallings folds and pass to the pointed core. The core graph of the full free group is the rose, so
\[
 \Gamma_X(F(X))=R_X.
\]
For a finite connected graph,
\[
 \rk(\Gamma)=\rk\pi_1(\Gamma)
 =|E^+\Gamma|-|V\Gamma|+1,
\]
where $E^+\Gamma$ contains one orientation of each geometric edge.

\subsection{Quotients, identifications, and free factors}\label{ss:quotients}
If $H\le J\le F(X)$, foldedness gives a unique pointed label-preserving morphism
\[
 \eta_{H\to J}:\Gamma_X(H)\longrightarrow\Gamma_X(J).
\]
Following Puder and Parzanchevski, write $H\le_X J$ when this morphism is surjective on vertices and edges. In this situation $\Gamma_X(J)$ is called an $X$-quotient of $\Gamma_X(H)$.

In the terminology of Miasnikov--Ventura--Weil~\cite[Section~2]{MiasnikovVenturaWeil},
these $J$ are the \emph{$X$-principal overgroups} of $H$; their collection
is the \emph{$X$-fringe} of $H$. That paper also relates this terminology
to the earlier ``overgroups'' of Margolis--Sapir--Weil and Ventura's
``fringe.'' For finitely generated $H$, this is a finite, basis-dependent
collection, not the set of all overgroups of $H$.

\begin{defn}[immediate quotient]\label{d:immediate}
Let $\Gamma$ be a finite pointed core graph over $X$. An \emph{immediate quotient} of $\Gamma$ is obtained by choosing two distinct vertices of $\Gamma$, identifying them, and then performing all forced Stallings folds and passing to the resulting core graph. We write
\[
 \Gamma\immq\Delta
\]
when $\Delta$ is an immediate quotient of $\Gamma$.

For $H\le_X J$, let $\rho_X(H,J)$ be the minimum length of a chain of immediate quotients from $\Gamma_X(H)$ to $\Gamma_X(J)$.
\end{defn}

An immediate quotient is the \emph{$i$-step} of
Silva--Weil~\cite[Section~2]{SilvaWeil}: identify two vertices and reduce
the resulting inverse automaton. Equivalently, it is a pinch followed by
folding and passage to the pointed core. We use Puder's immediate-quotient
notation to keep track of the minimum number of these operations.

The following no-leaves observation is also recorded, for folded graphs,
in Ciobanu--Elder--Ferov~\cite[Lemma~11]{CiobanuElderFerov}.

\begin{lem}[persistence of the no-leaves property]\label{l:no-leaves}
Let $\Gamma$ be a finite connected pointed $X$-graph such that every vertex has at least two outgoing oriented edges with distinct labels in $\Sigma_X$. Identify any two vertices and perform all forced Stallings folds. Then the resulting folded quotient has the same property. In particular, it has no vertex of degree one, passage to its pointed core deletes nothing, and the core graph produced by an immediate quotient again has the same property. The conclusion persists through any sequence of immediate quotients.
\end{lem}

\begin{proof}
Let $q:\Gamma\to\Delta$ be the quotient map obtained from the vertex identification and all subsequent folds, before any core pruning. For a vertex $v'\in V\Delta$, choose a vertex $v\in V\Gamma$ with $q(v)=v'$. By hypothesis there are outgoing edges $e,f$ at $v$ with $\mu(e)\ne\mu(f)$. A Stallings fold can identify only edges having the same label. Hence $q(e)$ and $q(f)$ remain distinct outgoing edges at $v'$, with their original distinct labels. Thus every vertex of $\Delta$ has degree at least two. The pointed core is obtained by pruning hanging degree-one vertices away from the basepoint, so no pruning occurs. Iteration proves the last assertion.
\end{proof}

Write $H\ff J$ when $H$ is a free factor of $J$, meaning that a free
basis of $H$ extends to one of $J$. The next proposition is Puder's
quotient-distance formulation for general $H\le_X J$
\cite[Theorem~1.1]{Puder}; see also~\cite[Theorem~3.8]{PuderParzanchevski}.

\begin{prop}[Puder]\label{p:puder}
Let $H\le_X J$ be finitely generated subgroups of $F(X)$. Then
\[
 H\ff J
 \quad\Longleftrightarrow\quad
 \rho_X(H,J)=\rk(J)-\rk(H).
\]
Moreover,
\[
 \rho_X(H,J)\ge \rk(J)-\rk(H).
\]
\end{prop}

The inequality follows because one immediate quotient increases rank by
at most one. The special case in which the target is the rose on the letters
occurring in $\Gamma_X(H)$ is already
Silva--Weil~\cite[Corollary~3.4]{SilvaWeil}: the required number of
$i$-steps is the support size minus $\rk(H)$. This is the case used in
Section~\ref{s:certificate}. A related criterion for primitive sets is
Clifford--Goldstein~\cite[Theorem~4.4]{CliffordGoldstein}; see also its
formulation in~\cite[Lemma~10]{CiobanuElderFerov}.

The algebraic effect of a vertex identification is
Silva--Weil~\cite[Proposition~2.4]{SilvaWeil}. We record it in our notation.

\begin{lem}[algebraic effect of an identification]\label{l:algebraic-identification}
Let $H\le F(X)$ be finitely generated, let $\Gamma=\Gamma_X(H)$, and
let $u,v$ be vertices of $\Gamma$. Choose paths from the basepoint to $u$ and
$v$ with labels $b_u,b_v\in F(X)$. Let $\Delta$ be obtained by identifying
$u$ and $v$, performing all forced Stallings folds, and passing to the pointed
core. Then
\[
 \Delta=\Gamma_X\big(\langle H,b_u b_v^{-1}\rangle\big).
\]
In particular, if $K$ is represented by the current graph in a sequence of
quotients, identifying the images of two vertices reached by words $b_u$ and
$b_v$ replaces $K$ by $\langle K,b_u b_v^{-1}\rangle$.
\end{lem}

\begin{proof}
If $u=v$, then $b_u b_v^{-1}\in H$ and the graph is unchanged, so the
assertion holds. Assume $u\ne v$. Before folding, the quotient map that
identifies $u$ and $v$ adds to
the image of the fundamental group one loop obtained by traveling from the
basepoint to $u$, crossing the identified point, and returning from $v$ to
the basepoint. Its label is $b_u b_v^{-1}$. Hence the subgroup represented by
the quotient is $\langle H,b_u b_v^{-1}\rangle$. This subgroup is independent
of the chosen paths: replacing the two path labels replaces $b_u b_v^{-1}$
by an element of the form $h_1(b_u b_v^{-1})h_2$, with $h_1,h_2\in H$, which
does not change the subgroup generated together with $H$. Stallings folds do
not change the represented subgroup, and passage to the pointed core gives
its core graph. This is also the algebraic description immediately following
Definition~3.5 in~\cite{PuderParzanchevski}.
\end{proof}

For comparison with~\cite[Section~3.3]{MiasnikovVenturaWeil}, an extension
$H\le K$ is \emph{elementary} if $K=\langle H,x\rangle$ for one element
$x$, and is \emph{algebraic} if no proper free factor of $K$ contains $H$.
For finitely generated $H$, an elementary extension either increases rank
by one, in which case $H$ is a proper free factor of $K$, or is algebraic.
An \emph{e-algebraic extension} is a finite composite of algebraic elementary
extensions. Thus immediate quotients are elementary extensions, but need
not be e-algebraic: the rank-increasing steps of the certificates in
Theorem~\ref{t:prefix-certificate} are free extensions.

\subsection{Reduction to the support}\label{ss:support}
\begin{conv}
Let $1\ne w\in F_r$, and let
\[
 I=\xsupp_{X_r}(w)=\{j\in\{1,\dots,r\}:x_j \text{ or } x_j^{-1}\text{ occurs in the reduced word for } w\}.
\]
Put $X_I=\{x_j:j\in I\}$ and write $F[I]=F(X_I)\le F_r$.
Thus $F[I]$ is a free factor of $F_r$ containing $w$.
\end{conv}

We need the following basic algebraic fact:
\begin{lem}[support reduction]\label{l:support}
Let $1\ne w\in F_r$, and let $I=\xsupp_{X_r}(w)$.

Then $w$ is primitive in $F_r$ if and only if $w$ is primitive in $F[I]$.
\end{lem}
\begin{proof}
A basis of $F[I]$ containing $w$ extends to a basis of $F_r$, proving one
direction. The converse is the support-reduction statement
in~\cite[Lemma~10.7]{KapovichMyasnikov}.
\end{proof}

\section{The prefix-difference certificate}\label{s:certificate}

We now express the free-factor criterion using positions in the original
cyclic word. The purpose is to obtain a certificate that can be recorded
without listing the vertices of its possibly exponentially large cycle.

\begin{conv}\label{conv:support}
For this section, let
\[
 w=y_1\cdots y_N\in\Sigma_r^*
\]
be nonempty, freely reduced, and cyclically reduced. Let
\[
 I=\xsupp_{X_r}(w),\qquad s=|I|,
\]
and work over the basis $X_I$. The pointed core graph of $\langle w\rangle$ is the directed cyclic graph $S_w$ spelling $w$. Number its vertices
\[
 v_0,v_1,\dots,v_{N-1}
\]
so that the directed edge from $v_{t-1}$ to $v_t$ has label $y_t$, with indices modulo $N$. Put
\[
 a_0=\eps,
 \qquad
 a_t=y_1\cdots y_t\quad(1\le t\le N).
\]
Thus the directed path in $S_w$ from the basepoint $v_0$ to $v_t$ has label $a_t$ for $0\le t<N$. When taking products in $F[I]$, we also write $a_t$ for the element represented by this prefix.
\end{conv}

For $0\le p,q<N$, we call $a_pa_q^{-1}$ a \emph{prefix difference}.
It is the product of one prefix with the inverse of another, regarded as
an element of $F_r$; the displayed terminal word need not be freely
reduced. By Lemma~\ref{l:algebraic-identification}, it is precisely the
element adjoined when the corresponding vertices are identified.

At $v_0$ the two outgoing oriented edges have labels $y_1$ and $y_N^{-1}$. For $1\le t<N$, the two outgoing labels at $v_t$ are $y_{t+1}$ and $y_t^{-1}$. Free reduction gives $y_{t+1}\ne y_t^{-1}$ for $1\le t<N$, and cyclic reduction gives $y_1\ne y_N^{-1}$ at the cyclic junction. Thus $S_w$ is folded and every vertex has two outgoing edges with distinct labels. Lemma~\ref{l:no-leaves} therefore implies that every graph occurring in any sequence of immediate quotients beginning at $S_w$ has no degree-one vertices; in these sequences the instruction to pass to the core does not delete any further edges.

Since every $x_j$ with $j\in I$ occurs in $w$ with one of its two signs, the canonical morphism
\[
 S_w=\Gamma_{X_I}(\langle w\rangle)\longrightarrow R_{X_I}=\Gamma_{X_I}(F[I])
\]
is surjective. Therefore
\[
 \langle w\rangle\le_{X_I}F[I],
\]
and Proposition~\ref{p:puder} applies.

\begin{lem}\label{l:sequential-subgroups}
Choose pairs
\[
 (p_1,q_1),\dots,(p_m,q_m),
 \qquad 0\le p_i,q_i<N,
\]
and define
\[
 u_i=a_{p_i}a_{q_i}^{-1},
 \qquad
 K_i=\langle w,u_1,\dots,u_i\rangle,
 \qquad K_0=\langle w\rangle.
\]
Starting with $S_w$, at stage $i$ identify the images of $v_{p_i}$ and $v_{q_i}$ in the current graph and perform all forced Stallings folds. Then the graph obtained after stage $i$ is $\Gamma_{X_I}(K_i)$.
\end{lem}

\begin{proof}
The assertion is clear for $i=0$. Suppose it holds through stage $i-1$. In $\Gamma_{X_I}(K_{i-1})$, the images of the paths from $v_0$ to $v_{p_i}$ and $v_{q_i}$ still have labels $a_{p_i}$ and $a_{q_i}$. Lemma~\ref{l:algebraic-identification} says that identifying their terminal vertices and folding changes the represented subgroup from $K_{i-1}$ to
\[
 \langle K_{i-1},a_{p_i}a_{q_i}^{-1}\rangle=K_i.
\]
Induction completes the proof.
\end{proof}

The following theorem reformulates the existing identification criterion
for compressed verification. Its useful feature here is that all positions
refer to the original cycle and only the final subgroup equality must be
tested.

\begin{thm}[prefix-difference certificate]\label{t:prefix-certificate}
Let $w=y_1\cdots y_N$ be a nonempty freely and cyclically reduced word over $\Sigma_r$. Let $I$, $s$, $F[I]$, and the prefixes $a_t$ be as in Convention~\ref{conv:support} above. Then the following are equivalent.
\begin{enumerate}
\item The element represented by $w$ is primitive in $F_r$.
\item There exist $s-1$ pairs of integers
\[
 (p_1,q_1),\dots,(p_{s-1},q_{s-1}),
 \qquad 0\le p_i<q_i<N,
\]
such that
\begin{equation}\label{eq:generation-certificate}
 F[I]=
 \left\langle
 w,
 a_{p_1}a_{q_1}^{-1},\dots,
 a_{p_{s-1}}a_{q_{s-1}}^{-1}
 \right\rangle.
\end{equation}
\end{enumerate}
When these conditions hold, the pairs may be chosen so that sequentially identifying the images of $v_{p_i}$ and $v_{q_i}$ produces a chain of exactly $s-1$ immediate quotients from $S_w$ to $R_{X_I}$.
\end{thm}

\begin{proof}
The restriction $p_i<q_i$ is only a normalization. Indeed, if
$p_i=q_i$ in~\eqref{eq:generation-certificate}, then the right-hand side is
generated by at most $s-1$ elements and cannot equal the group $F[I]$ of rank $s$; interchanging $p_i$ and $q_i$ only inverts the corresponding
generator.

By Lemma~\ref{l:support}, condition (a) is equivalent to $\langle w\rangle\ff F[I]$. Since $\langle w\rangle\le_{X_I}F[I]$, Proposition~\ref{p:puder} gives
\begin{equation}\label{eq:puder-specialized}
 \langle w\rangle\ff F[I]
 \quad\Longleftrightarrow\quad
 \rho_{X_I}(\langle w\rangle,F[I])
 =\rk(F[I])-\rk(\langle w\rangle)=s-1.
\end{equation}

Suppose first that (a) holds. By~\eqref{eq:puder-specialized}, there is a chain
\begin{equation}\label{eq:chain}
 S_w=\Gamma_0\immq\Gamma_1\immq\cdots
 \immq\Gamma_{s-1}=R_{X_I}.
\end{equation}
At the $i$-th step, let $\bar u_i,\bar v_i$ be the two distinct vertices of $\Gamma_{i-1}$ that are identified. The composite quotient map
\[
 S_w\longrightarrow\Gamma_{i-1}
\]
is surjective on vertices. Choose one preimage of each. They are distinct,
and after relabeling them we may write them as $v_{p_i},v_{q_i}$ with
$0\le p_i<q_i<N$. Lemma~\ref{l:sequential-subgroups}, applied inductively
along~\eqref{eq:chain}, shows that
\[
 \Gamma_i=
 \Gamma_{X_I}\left(
 \left\langle w,
 a_{p_1}a_{q_1}^{-1},\dots,
 a_{p_i}a_{q_i}^{-1}
 \right\rangle
 \right).
\]
At $i=s-1$ the graph is $R_{X_I}$, so the represented subgroup is $F[I]$. Thus~\eqref{eq:generation-certificate} holds, and the final assertion about the chain follows from the construction.

Conversely, suppose that (b) holds, and put
$u_i=a_{p_i}a_{q_i}^{-1}$. Fix an ordering
$X_I=(x_{j_1},\dots,x_{j_s})$ and define $\theta\in\operatorname{End}(F[I])$ by
$\theta(x_{j_1})=w$ and $\theta(x_{j_{i+1}})=u_i$ for $1\le i\le s-1$.
Equation~\eqref{eq:generation-certificate} gives
$\theta(F[I])=\langle w,u_1,\dots,u_{s-1}\rangle=F[I]$, so $\theta$ is
surjective. Finitely generated free groups are Hopfian, hence
$\theta$ is an automorphism. Therefore $(w,u_1,\dots,u_{s-1})$ is a free
basis of $F[I]$, and in particular $w$ is primitive in $F[I]$.
Lemma~\ref{l:support} implies that $w$ is primitive in $F_r$.

It remains to justify the asserted sequential interpretation. Apply the indicated $s-1$ vertex-pair identifications in order. By Lemma~\ref{l:sequential-subgroups}, the final graph is
\[
 \Gamma_{X_I}(F[I])=R_{X_I}.
\]
If the two selected vertices already have the same image at some stage, that stage is a no-op rather than an immediate quotient. Deleting all no-op stages gives a chain of immediate quotients of length at most $s-1$. The rank lower bound in Proposition~\ref{p:puder} gives
\[
 \rho_{X_I}(\langle w\rangle,F[I])
 \ge \rk(F[I])-\rk(\langle w\rangle)=s-1.
\]
Consequently no stage is a no-op, and the original sequence is a chain of exactly $s-1$ immediate quotients.

\end{proof}

\begin{cor}[full-support form]\label{c:full-support}
Suppose that $w$ involves every generator $x_1,\dots,x_r$. Then $w$ is primitive in $F_r$ if and only if there exist $r-1$ pairs of vertices of $S_w$ such that sequentially identifying their images and folding produces the rose $R_{X_r}$ after exactly $r-1$ immediate quotients. Equivalently, there are positions $(p_i,q_i)_{i=1}^{r-1}$ with
$0\le p_i<q_i<N$ such that
\[
 F_r=\left\langle
 w,a_{p_1}a_{q_1}^{-1},\dots,
 a_{p_{r-1}}a_{q_{r-1}}^{-1}
 \right\rangle.
\]
\end{cor}

The case $s=1$ is included. There are no guessed pairs, and the certificate condition is simply $\langle w\rangle=F[I]$, which is equivalent to $w=x_j^{\pm1}$ for the unique $j\in I$.

\section{Compressed primitivity in fixed rank}\label{s:verification}

We prove Theorem~\ref{t:main} by giving a polynomial-size certificate for
each primitive input and a deterministic polynomial-time verifier. The
certificate is supplied by Theorem~\ref{t:prefix-certificate}; the main
algorithmic ingredient for checking it is compressed subgroup membership.

\subsection{The compressed membership algorithm}\label{ss:membership}
We use the following consequence of Linton's algorithm
\cite[Theorem~5.8 and Corollary~5.9]{Linton}.

\begin{prop}[Linton]\label{p:linton}
For fixed $r\ge2$ and $k\ge1$, there is a deterministic polynomial-time
algorithm which, given $\SLP$s $\mathbb W,\mathbb W_1,\dots,\mathbb W_k$
over $\Sigma_r$, decides whether the element represented by $\mathbb W$
belongs to
\[
 H=\langle\val(\mathbb W_1),\dots,\val(\mathbb W_k)\rangle\le F_r.
\]
If $M$ is the total structural size of these $\SLP$s, the running-time
bound has the form $M^{O(k)}$.
\end{prop}

\begin{proof}
Linton~\cite[Theorem~5.8]{Linton} gives a bound of the form $n^{O(k)}$
in his input-size notation, and~\cite[Corollary~5.9]{Linton} gives
polynomial time for fixed $k$. For total encoded input length $B$ in our
encoding, the bound is $B^{O(k)}$ after the standard change of encoding.
By Proposition~\ref{p:slp-size-comparison}, $B=O(M\log(M+2))$; absorbing
the logarithmic factor into the exponent gives $M^{O(k)}$.
\end{proof}

In the graph language of Section~\ref{s:core}, a compressed graph stores
transition labels using a shared acyclic grammar rather than expanding
them into individual letters. A transition $p\xrightarrow{A_i}q$
represents a path labeled $w_{A_i}$, and its reverse carries $w_{A_i}^{-1}$.
The size includes the finite state graph and the shared grammar, not the
possibly exponential number of edges in its subdivision. Placing the
subgroup generators on compressed loops at a single base state gives the
input to Linton's compressed folding procedure. The procedure constructs
a compressed deterministic automaton for the subgroup and then tests the
compressed query word without expanding the labels. The input words may
be freely reduced first by Proposition~\ref{p:cyclic-reduction}(a), without
changing the polynomial bound for fixed $k$.

\subsection{Binary certificates and their verification}\label{ss:verifier}
The next observation explains why vertices of the original cycle can be
named economically, even when the cycle itself cannot be listed.

\begin{lem}[binary size of vertex positions]\label{l:certificate-size}
Let $\CB$ be an $\SLP$ of size $m$ with nonempty output $w$ of length $N$.
Each position $t$ with $0\le t<N$, written in binary without leading zeroes,
has $O(m)$ bits. For fixed $r$, a list of at most $r-1$ pairs
$(p_i,q_i)$ with $0\le p_i<q_i<N$ therefore has $O(m)$ bits.
\end{lem}

\begin{proof}
After polynomial-time normalization, Lemma~\ref{l:slp-ops}(a) gives
$N\le2^{O(m)}$. Hence every position has $O(m)$ bits. The list contains at
most $2(r-1)$ positions, and $r$ is fixed. Since $m\le\beta(\CB)$, the same
conclusion holds with respect to the encoded length of $\CB$.
\end{proof}

\begin{prop}[the verifier]\label{p:verifier}
For each fixed $r\ge2$, there are a polynomial $p_r$ and a deterministic
verifier $V_r$ such that, for every $\SLP$ $\CA$ over $\Sigma_r$,
\[
 g_{\CA}\text{ is primitive in }F_r
 \quad\Longleftrightarrow\quad
 \exists C\ \bigl(|C|\le p_r(\beta(\CA))
       \text{ and }V_r(\enc_r(\CA),C)=1\bigr).
\]
The verifier runs in time polynomial in $\beta(\CA)+|C|$. It rejects
invalid $\SLP$ encodings and malformed certificates; certificates encode
lists of pairs of binary positions.
\end{prop}

\begin{proof}
First parse the purported $\SLP$ encoding and reject invalid encodings.
For a valid input $\CA$, put $n=\lVert\CA\rVert$. The verifier performs
the following steps on this input and the supplied certificate $C$.

\smallskip
\noindent\emph{Step 1: cyclic reduction.}
Using Proposition~\ref{p:cyclic-reduction}(b), construct in deterministic polynomial time an $\SLP$ $\CB$ whose output
\[
 w=\val(\CB)
\]
is freely and cyclically reduced and represents a conjugate of $g_{\CA}$. Let $m=\lVert\CB\rVert$ and compute $N=|w|$ in binary. Since $m$ is polynomially bounded in $n$, all later polynomial bounds in $m$ are polynomial bounds in $n$. If $N=0$, reject.

\smallskip
\noindent\emph{Step 2: support.}
Compute
\[
 I=\xsupp_{X_r}(w),\qquad s=|I|
\]
using Lemma~\ref{l:slp-ops}(d). The verifier requires the certificate to contain exactly $s-1$ pairs
\[
 (p_1,q_1),\dots,(p_{s-1},q_{s-1})
\]
and checks that $0\le p_i<q_i<N$ for every $i$. Positions are written in
binary without leading zeroes, with $0$ represented by one zero bit; fixed
delimiters separate entries. Reject any malformed list. By
Lemma~\ref{l:certificate-size}, a valid certificate has $O(m)$ bits. Since
$m$ is polynomially bounded in $n\le\beta(\CA)$, this gives the polynomial
$p_r$ required in the statement.

\smallskip
\noindent\emph{Step 3: compressed prefix differences.}\par\noindent
For every position $t$ occurring in the certificate, use Lemma~\ref{l:slp-ops}(b) to construct an $\SLP$ $\CP_t$ satisfying
\[
 \val(\CP_t)=a_t=w[0:t].
\]
Use Lemma~\ref{l:slp-ops}(c) and one new concatenation production to construct an $\SLP$ $\mathbb U_i$ with
\[
 \val(\mathbb U_i)=a_{p_i}a_{q_i}^{-1}.
\]
There are at most $2(r-1)$ prefix $\SLP$s and $r-1$ difference $\SLP$s. Their total size and the time needed to construct them are polynomial in $m$.

\smallskip
\noindent\emph{Step 4: endpoint subgroup.}
Let
\[
 K=\left\langle
 w,\val(\mathbb U_1),\dots,\val(\mathbb U_{s-1})
 \right\rangle.
\]
This subgroup is given by exactly $s\le r$ $\SLP$ generators of total size
polynomial in $m$. For every $j\in I$, use the one-letter $\SLP$ for $x_j$ as
the query in Proposition~\ref{p:linton} and decide whether $x_j\in K$.
Accept precisely when all these membership tests return yes.

Every generator displayed for $K$ is a word over $X_I^{\pm1}$, so $K\le F[I]$. Therefore
\[
 x_j\in K\text{ for all }j\in I
 \quad\Longleftrightarrow\quad
 K=F[I].
\]
The number of subgroup generators is $s\le r$. By Proposition~\ref{p:linton}, each membership test has running time
$M^{O(s)}$, where $M$ is the total structural size of that membership
instance and is polynomial in $m$. Since $r$ is fixed, this is a polynomial bound with a constant exponent; there are at most $r$ tests. Thus $V_r$ runs in deterministic polynomial time.

It remains to prove correctness. Suppose first that $g_{\CA}$ is primitive. Its cyclic reduction $w$ is primitive, and Theorem~\ref{t:prefix-certificate} supplies $s-1$ pairs for which $K=F[I]$. The corresponding certificate is accepted.

Conversely, if a certificate is accepted, the basis-membership tests show that $K=F[I]$. Equation~\eqref{eq:generation-certificate} therefore holds, so Theorem~\ref{t:prefix-certificate} implies that $w$ is primitive in $F_r$. Since $w$ is conjugate to $g_{\CA}$, the element $g_{\CA}$ is primitive as well.
\end{proof}

We can now prove the first main result of this paper, Theorem~\ref{t:main} from the Introduction.

\begin{thm}\label{t:main1}
For every fixed integer $r\ge 2$, the language $\CPrim_r$ belongs to $\mathsf{NP}$.
\end{thm}

\begin{proof}
First parse the input according to Definition~\ref{d:slp-encoding} and reject
it if it is not a valid $\SLP$ encoding.  For a valid encoding,
Proposition~\ref{p:verifier} gives a certificate of polynomial bit-size for
every yes-instance, and a deterministic polynomial-time algorithm checks the
certificate. Hence $\CPrim_r\in\mathsf{NP}$.
\end{proof}

\subsection{Endpoint verification and dependence on rank}\label{s:remarks}

\begin{rem}[endpoint verification and the sequential interpretation]
The verifier checks only the endpoint equality $K=F[I]$ and does not explicitly construct the intermediate core graphs.  Nevertheless, if
\[
 K_i=\langle w,u_1,\dots,u_i\rangle,
 \qquad u_i=a_{p_i}a_{q_i}^{-1},
\]
then Lemma~\ref{l:sequential-subgroups} identifies $\Gamma_{X_I}(K_i)$ with
the graph obtained after the first $i$ prescribed vertex identifications.
For an accepted certificate, $K_i$ is generated by $i+1$ elements and
$F[I]=\langle K_i,u_{i+1},\dots,u_{s-1}\rangle$. Hence
\[
 \rk(K_i)\le i+1,\qquad
 s\le\rk(K_i)+(s-1-i),
\]
so $\rk(K_i)=i+1$ for $0\le i\le s-1$. Thus no stage is a no-op, and
\[
 S_w=\Gamma_{X_I}(K_0)\immq\Gamma_{X_I}(K_1)
 \immq\cdots\immq\Gamma_{X_I}(K_{s-1})=R_{X_I}
\]
is a chain of exactly $s-1$ immediate rank-increasing quotients.  Thus there is no need to construct the intermediate graphs or to run Linton's algorithm separately on each $K_i$: verifying the endpoint certifies the entire chain.
\end{rem}

\begin{rem}[bounded rank does not bound graph size]
The rank bound alone is not a sufficient complexity argument. The initial graph $S_w$ has rank one but $N$ vertices, where $N$ may be exponential in the input $\SLP$ size. The prefix-difference formulation avoids ever listing these vertices or edges. A vertex is named by an $O(\log(N+1))$-bit position, and its path from the basepoint is represented by a polynomial-size prefix $\SLP$.
\end{rem}

\begin{rem}[dependence on the ambient rank]
The proof establishes an $\mathsf{NP}$ upper bound for each fixed $r$.
Linton's quantitative bound is $M^{O(k)}$, where $M$ is the total
structural size of the membership instance and the subgroup is given by $k$
compressed generators. Here $k=s\le r$, so the verifier has a bound of the
form $M^{O(r)}$, and the exponent is constant once $r$ is fixed. The constant
hidden in this exponent is inherited from Linton's analysis and is not made
explicit here. The exponent in the running-time bound depends on $r$.
Thus this argument does not give a uniform polynomial-time verifier when
$r$ is part of the input.
\end{rem}

\section{Compressed Whitehead minimality}
\label{s:whitehead-minimization}

For a conjugacy class $[g]$ in $F_r$, let
\[
 \lVert[g]\rVert_{X_r}
 =\min\{|u|:u\in\Sigma_r^*\text{ is freely reduced and represents an
 element conjugate to }g\}.
\]
Equivalently, if $u$ is any word representing $g$, then
$\lVert[g]\rVert_{X_r}=|\cycred(u)|$.  The automorphism group acts on
conjugacy classes by $\phi([g])=[\phi(g)]$.  We say that $[g]$ is
\emph{automorphically minimal} if
\[
 \lVert\phi([g])\rVert_{X_r}\ge \lVert[g]\rVert_{X_r}
 \qquad\text{for every }\phi\in\Aut(F_r).
\]
This section is independent of the certificate proof. We first recall
Whitehead's automorphisms, weighted graph, and length-change formula in
Section~\ref{ss:whitehead-background}, then compute the graph from an
$\SLP$ and test minimality. The final subsection discusses evaluating
prescribed sequences of automorphisms and why this does not by itself
solve compressed minimization.

\subsection{Whitehead automorphisms and weighted graphs}\label{ss:whitehead-background}

A \emph{Whitehead automorphism of the first kind} is a relabeling
automorphism: its restriction to $\Sigma_r$ is a permutation commuting with
inversion.  Such an automorphism preserves cyclic length.  A
\emph{Whitehead automorphism of the second kind} is specified by a
\emph{characteristic pair} $(A,a)$, where
\[
 A\subseteq\Sigma_r,\qquad a\in A,\qquad a^{-1}\notin A.
\]
It fixes $a^{\pm1}$ and, for
$x\in\Sigma_r\setminus\{a,a^{-1}\}$, is given by
\begin{equation}\label{eq:whitehead-characteristic-pair}
 \tau_{(A,a)}(x)=
 \begin{cases}
 x, &x\notin A\text{ and }x^{-1}\notin A,\\
 xa, &x\in A\text{ and }x^{-1}\notin A,\\
 a^{-1}x, &x\notin A\text{ and }x^{-1}\in A,\\
 a^{-1}xa, &x\in A\text{ and }x^{-1}\in A.
 \end{cases}
\end{equation}
These two kinds form the finite set $\mathfrak W_r$ of Whitehead
automorphisms relative to the basis $X_r$.  There are at most
$2r\,4^{r-1}$ characteristic pairs of the second kind, a constant when $r$
is fixed.

A conjugacy class is \emph{Whitehead minimal} if no automorphism in
$\mathfrak W_r$ decreases its cyclic length.  Since automorphisms of the
first kind preserve cyclic length, it is enough to consider the
characteristic pairs above.

\begin{prop}[Whitehead's length-reduction theorem]
\label{p:whitehead-reduction}
A conjugacy class in $F_r$ is automorphically minimal if and only if it is
Whitehead minimal.  More precisely, if a nontrivial cyclically reduced word
$w$ does not represent an automorphically minimal conjugacy class, then
there is a Whitehead automorphism $\tau$ of the second kind such that
\[
 \lVert\tau([w])\rVert_{X_r}<|w|.
\]
\end{prop}

\begin{proof}
The forward implication is immediate.  The reverse implication is the
length-reduction part of Whitehead's theorem; see
Whitehead~\cite{Whitehead} and
Kapovich--Schupp--Shpilrain~\cite[Proposition~1.2(1)]{KapovichSchuppShpilrain}.
The trivial conjugacy class, of cyclic length zero, is minimal under every
automorphism.
\end{proof}

Let $w=z_1\cdots z_N$ be a nontrivial cyclically reduced word.  For
$u,v\in\Sigma_r$, define the number of cyclic occurrences of the
length-two word $uv$ by
\[
 n_w^\circ(uv)
 =\#\{i\in\mathbb Z/N\mathbb Z:z_i=u,\ z_{i+1}=v\},
\]
where subscripts are read modulo $N$.  Thus the pair formed by the last and
first letters is included.

\begin{defn}[weighted Whitehead graph]
\label{d:weighted-whitehead-graph}
The \emph{weighted Whitehead graph} $\mathcal W_w$ has vertex set
$\Sigma_r$.  For every unordered pair $\{p,q\}$ of distinct vertices, it
has one undirected edge of weight
\begin{equation}\label{eq:whitehead-edge-weight}
 \omega_w(\{p,q\})
 =n_w^\circ(pq^{-1})+n_w^\circ(qp^{-1}).
\end{equation}
Edges of weight zero are retained.  Hence $\mathcal W_w$ has
$\binom{2r}{2}=r(2r-1)$ edges.  Cyclic permutation of $w$, and replacement
of $w$ by $w^{-1}$, do not change the graph, so $\mathcal W_w$ depends only
on the unoriented cyclic word, and in particular only on the represented
conjugacy class together with the basis $X_r$.

For disjoint subsets $P,Q\subseteq\Sigma_r$, write
\[
 P\mathbin{\cdot_w}Q
 =\sum_{\substack{p\in P,\ q\in Q}}
   \omega_w(\{p,q\}),
\]
and put
\[
 \deg_w(a)=\{a\}\mathbin{\cdot_w}
            (\Sigma_r\setminus\{a\}).
\]
Thus $\deg_w(a)$ is the weighted degree of the vertex $a$; equivalently, it
is the total number of occurrences of $a^{\pm1}$ in $w$.
\end{defn}

A cyclic occurrence of $uv$ therefore contributes to the edge
$\{u,v^{-1}\}$. This convention matches
\eqref{eq:whitehead-characteristic-pair}. It differs from the graph in
Kapovich--Schupp--Shpilrain~\cite[Definition~4.5]{KapovichSchuppShpilrain}
by relabeling each vertex with its inverse.

\begin{prop}[Whitehead's length-change formula]
\label{p:whitehead-length-formula}
Let $w$ be a nontrivial cyclically reduced word, and let
$\tau=\tau_{(A,a)}$ be a Whitehead automorphism of the second kind.  Put
$A^c=\Sigma_r\setminus A$.  Then
\begin{equation}\label{eq:whitehead-length-formula}
 \lVert\tau([w])\rVert_{X_r}-|w|
 =A\mathbin{\cdot_w}A^c-\deg_w(a).
\end{equation}
\end{prop}

\begin{proof}
This is Whitehead's classical length formula; compare
\cite[Lemma~4.7]{KapovichSchuppShpilrain} and
\cite[Proposition~2.4]{RoigVenturaWeil}. We verify it directly to make the
conventions explicit. If $w$ uses only $a^{\pm1}$, then $\tau$ fixes $w$
and both $A\mathbin{\cdot_w}A^c$ and $\deg_w(a)$ equal $|w|$.
Otherwise, write $w$ cyclically as
$x_1a^{k_1}\cdots x_ta^{k_t}$, where $x_i\in\Sigma_r\setminus\{a,a^{-1}\}$ and
$k_i\in\mathbb Z$. Put $\epsilon_i=1$ if $x_i\in A$ and $0$ otherwise,
and $\eta_i=1$ if $x_i^{-1}\in A$ and $0$ otherwise. Indices are cyclic.
By \eqref{eq:whitehead-characteristic-pair}, a cyclic representative of
$\tau([w])$ is obtained by replacing each exponent $k_i$ with
$k_i+\epsilon_i-\eta_{i+1}$. No $x_i$ cancels: if
$x_{i+1}=x_i^{-1}$, then $\epsilon_i=\eta_{i+1}$ and $k_i\ne0$.
Hence the cyclic length change is
\[
 \sum_{i=1}^t
 \bigl(|k_i+\epsilon_i-\eta_{i+1}|-|k_i|\bigr).
\]
The number of cut crossings contributed by the transitions in
$x_i a^{k_i}x_{i+1}$ is $k_i+\epsilon_i-\eta_{i+1}$ when $k_i>0$,
$-k_i-\epsilon_i+\eta_{i+1}$ when $k_i<0$, and
$|\epsilon_i-\eta_{i+1}|$ when $k_i=0$. In all three cases it equals
$|k_i+\epsilon_i-\eta_{i+1}|$. Summing these contributions gives
$A\mathbin{\cdot_w}A^c$, while $\deg_w(a)=\sum_i|k_i|$, proving
\eqref{eq:whitehead-length-formula}.
\end{proof}

\subsection{Computing the graph and testing minimality}\label{ss:whitehead-test}

\begin{prop}[compressed computation of the weighted Whitehead graph]
\label{p:compressed-whitehead-graph}
Given an $\SLP$ $\CA$ over $\Sigma_r$, there exists a deterministic
polynomial-time algorithm with the following properties.
\begin{enumerate}
\item It constructs an $\SLP$ $\CB$ with
\[
 w_{\CB}=\cycred(w_{\CA})
\]
and, when this word is nonempty, computes the complete weighted Whitehead
graph $\mathcal W_{w_{\CB}}$ with binary edge weights.
\item Suppose that a Chomsky-normal-form $\SLP$ for the nonempty word
$w_{\CB}$ has $m$ nonterminals, size $M$, and no nonterminals unreachable
from the root. With $L=\blen(|w_{\CB}|)$, the cyclic two-letter counts and
graph weights can be computed using $O(mL)$, hence $O(M^2)$, bit
operations with schoolbook addition. Every weight has at most $L$ bits.
\end{enumerate}
\end{prop}

\begin{proof}
Construct $\CB$ by Proposition~\ref{p:cyclic-reduction}(b), and normalize
it. If necessary, delete all nonterminals that are not reachable from the root,
relabel the remaining variables in dependency order, and continue to denote the
resulting program by $\CB$. This does not change the output and requires
linearly many grammar-symbol operations. Process the remaining nonterminals
in dependency order.
For every nonterminal $D$, store the following data about its expansion $w_D$:
\begin{enumerate}
\item its length $\ell_D$ in binary;
\item when $\ell_D>0$, its first and last terminal letters;
\item for every ordered pair $(u,v)\in\Sigma_r^2$, the number
$c_D(u,v)$ of ordinary, noncyclic occurrences of $uv$ in $w_D$.
\end{enumerate}
The data are immediate for a terminal production and for the exceptional
empty root.  If $D\to EF$, then
\[
 \ell_D=\ell_E+\ell_F,
\]
the first and last letters are inherited from the first and last nonempty
factors, and
\begin{equation}\label{eq:pair-count-recursion}
 c_D(u,v)=c_E(u,v)+c_F(u,v)
 +\begin{cases}
 1,&\ell_E,\ell_F>0,\ \operatorname{last}(w_E)=u,
       \ \operatorname{first}(w_F)=v,\\
 0,&\text{otherwise}.
 \end{cases}
\end{equation}
Put $N=|w_{\CB}|$. Since every remaining nonterminal $D$ is reachable from
the root, one occurrence of its expansion appears as a factor of the root
expansion. Thus there are words $P_D,Q_D\in\Sigma_r^*$ such that, as literal
words, $w_{\CB}=P_Dw_DQ_D$. Consequently, for every
$(u,v)\in\Sigma_r^2$,
\[
 0\le \ell_D\le N,
 \qquad
 0\le c_D(u,v)\le\max\{\ell_D-1,0\}\le N.
\]
Since $r$ is fixed, only constantly many counters are stored per nonterminal.
Each counter has at most $L=\blen(N)$ bits, so the total bit cost is
$O(mL)$.

If $w_{\CB}$ is nonempty, then at the root add one to the count
corresponding to the ordered pair formed by the last and first letters. This
gives every cyclic count $n_{w_{\CB}}^\circ(uv)$. For distinct
$p,q\in\Sigma_r$, the ordered pairs $pq^{-1}$ and $qp^{-1}$ are distinct, so
their occurrence sets among the $N$ cyclic positions are disjoint. Hence
$0\le\omega_{w_{\CB}}(\{p,q\})\le N$ by
Formula~\eqref{eq:whitehead-edge-weight}, and every graph weight also has at
most $L$ bits. The root correction and all edge-weight computations use only
$O(L)$ additional bit operations. Since $m,L=O(M)$, the stated $O(M^2)$
bound follows.
The initial cyclic-reduction and normalization stages have polynomial cost
in $\lVert\CA\rVert$ by Section~\ref{s:slp}.
\end{proof}

\begin{thm}[compressed Whitehead-minimality test]
\label{t:compressed-whitehead-minimality}
For every fixed $r\ge2$, there is a deterministic polynomial-time algorithm
which, given an $\SLP$ $\CA$ over $\Sigma_r$, decides whether the conjugacy
class represented by $\cycred(w_{\CA})$ is Whitehead minimal.  Equivalently,
it decides whether this conjugacy class is automorphically minimal in
$F_r$.

Once a Chomsky-normal-form $\SLP$ $\CB$ of size $M$ for
$\cycred(w_{\CA})$ has been produced, the graph construction and the
minimality test together use $O(M^2)$ bit operations.
\end{thm}

\begin{proof}
Use Proposition~\ref{p:compressed-whitehead-graph}.  If $w_{\CB}=\eps$,
accept, since the trivial conjugacy class is automorphically minimal.
Otherwise compute $\mathcal W_{w_{\CB}}$ and enumerate all characteristic
pairs $(A,a)$ of Whitehead automorphisms of the second kind.  For each pair,
compute the binary integer
\[
 \Delta(A,a)=A\mathbin{\cdot_{w_{\CB}}}A^c-\deg_{w_{\CB}}(a).
\]
By Proposition~\ref{p:whitehead-length-formula}, $\Delta(A,a)$ is exactly
the change in cyclic length under $\tau_{(A,a)}$.  Accept if and only if
$\Delta(A,a)\ge0$ for every characteristic pair.  Automorphisms of the first
kind need not be checked because they preserve cyclic length.

For fixed $r$, the number of characteristic pairs and the number of edges
summed for each pair are constants.  The arithmetic after the graph is
constructed is therefore linear in the bit-length of its weights and is
dominated by the $O(M^2)$ graph computation.  Correctness follows from
Proposition~\ref{p:whitehead-reduction}.
\end{proof}

The proof deliberately avoids constructing $\tau(w_{\CB})$ for any
Whitehead automorphism.  When the rank is allowed to vary, enumerating all
characteristic pairs is exponential in $r$.  Roig--Ventura--Weil
\cite{RoigVenturaWeil} reinterpret
$A\mathbin{\cdot_w}A^c$ as a cut capacity and use minimum-cut algorithms to
find an optimal Whitehead move in time polynomial in both the ordinary word
length and the rank.  For the fixed-rank compressed problem considered
here, the constant-size enumeration above is sufficient.

\subsection{Why elementary descent can be too long}\label{ss:descent}

Classical Whitehead descent starts with a cyclically reduced representative,
applies a Whitehead automorphism that decreases its cyclic length, and
cyclically reduces the image. Repeating this procedure terminates at a
minimal representative by Proposition~\ref{p:whitehead-reduction}.

The layered construction of Lohrey--Schleimer~\cite{LohreySchleimer} and
Schleimer~\cite[Theorem~5.2 and Remark~5.3]{Schleimer} evaluates a prescribed
sequence of $m$ maps from a fixed finite family of automorphisms on an
input $\SLP$ in time polynomial in $m$ and the compressed input size.
Each stage stores the terminal images once, and free reduction may be
deferred until the end. This does not bound the number of descent stages.
Already in rank two, there is no constant $0<\delta<1$ such that every
nonminimal class has a Whitehead move reducing its cyclic length by a
factor at most $1-\delta$. The following example exhibits the obstruction.

\begin{exmp}\label{ex:slow-whitehead-descent} Let $n\ge1$, and
consider the $\SLP$
\[
 B_0\to b,\qquad
 B_i\to B_{i-1}B_{i-1}\quad(1\le i\le n),
 \qquad C\to aB_n.
\]
After an order-preserving relabeling of its nonterminals, this is an $\SLP$
with size $O(n)$ and encoded length $O(n\log(n+2))$ whose output
is
\[
 w_n=ab^{2^n}.
\]
Let $\tau\in\Aut(F(a,b))$ be the Whitehead automorphism
\[
 \tau(a)=ab^{-1},\qquad \tau(b)=b.
\]
For $N=2^n$ one has
\[
 \tau^k(ab^N)=ab^{N-k}\qquad(0\le k\le N).
\]
For $N\ge2$, the only nonzero edges of the weighted Whitehead graph of
$ab^N$ are
\[
 \{a,b^{-1}\},\qquad \{a^{-1},b\},\qquad \{b,b^{-1}\},
\]
with respective weights $1,1,N-1$.  Inspecting the characteristic pairs in
Whitehead's length-change formula, the only pairs with negative length change
are
\[
 (\{a^{-1},b\},b)
 \quad\text{and}\quad
 (\{a,b^{-1}\},b^{-1}).
\]
Each pair has length change $-1$, and the image represents the conjugacy
class $[ab^{N-1}]$. For $N=1$, one further reducing move reaches a one-letter
minimal class. Thus every strictly decreasing Whitehead descent from
$[ab^{2^n}]$ has $2^n$ stages, exponential in the grammar depth and
superpolynomial in both input size measures.
\end{exmp}

For the word in this example, the single automorphism
\[
 \tau_N(a)=ab^{-N},\qquad \tau_N(b)=b
\]
is the power $\tau^N$ and sends $ab^N$ directly to $a$. Such a power is
called a \emph{shear}: on exponent-sum pairs it acts by
$(p,q)\mapsto(p,q-Np)$. The power $b^N$ has an $\SLP$ of size
$O(\log(N+1))$. Section~\ref{s:rank-two} implements these batched moves
using binary parameters and compressed powers. Its normalization works on
all rank-two inputs with coprime exponent sums; literal shortening is
asserted only for primitive conjugacy classes. The fixed-family estimate
alone does not provide an analogous acceleration in higher rank.

\section{Deterministic polynomial time in rank two}\label{s:rank-two}

We now specialize to
\[
 F_2=F(a,b).
\]
In rank two the compressed primitivity problem admits a deterministic
polynomial-time algorithm. Its extra algebraic ingredient is the classical
fact that primitive elements with the same abelianization are conjugate.
We first state the algebraic facts, then describe Euclidean normalization,
its compressed realization, and the final decision test.

\subsection{Primitivity and rank-two automorphisms}\label{ss:rank-two-algebra}

For $g\in F_2$, write
\[
 \operatorname{ab}_2(g)=\big(\sigma_a(g),\sigma_b(g)\big)\in\mathbb Z^2,
\]
where $\sigma_a$ and $\sigma_b$ are the exponent sums in $a$ and $b$,
respectively.

We use the following form of the Nielsen--Osborne--Zieschang classification;
see~\cite{OsborneZieschang}.

\begin{prop}[Nielsen--Osborne--Zieschang]\label{p:oz-classification}
Let $g\in F_2$.
\begin{enumerate}
\item If $g$ is primitive, then
\[
 \gcd\big(|\sigma_a(g)|,|\sigma_b(g)|\big)=1.
\]
\item For every $(p,q)\in\mathbb Z^2$ with $\gcd(|p|,|q|)=1$, there is a
primitive element with exponent-sum pair $(p,q)$.
\item If $g,h\in F_2$ are primitive and
$\operatorname{ab}_2(g)=\operatorname{ab}_2(h)$, then $g$ and $h$ are conjugate.
\end{enumerate}
In particular, the map induced by abelianization gives a bijection from primitive
conjugacy classes in $F_2$ to primitive vectors in $\mathbb Z^2$.
\end{prop}

\begin{proof}
The classification of primitive elements and bases in rank two is due
to Nielsen and is given in the precise form used here by
Osborne--Zieschang~\cite{OsborneZieschang}. A related, somewhat less
precise structural description of bases in $F_2$ appears in
Cohen--Metzler--Zimmermann~\cite{CohenMetzlerZimmermann}; see also the
Christoffel-word treatment of
Kassel--Reutenauer~\cite{KasselReutenauer}. We recall the short deductions
needed below. Nielsen's rank-two theorem identifies
\[
 \Out(F_2)\cong\operatorname{GL}(2,\mathbb Z);
\]
equivalently, the homomorphism
$\Aut(F_2)\to\operatorname{GL}(2,\mathbb Z)$ induced by abelianization is surjective and has kernel $\operatorname{Inn}(F_2)$.

If $g$ is primitive, extend it to an ordered basis $(g,g_2)$. The two
abelianization vectors form a basis of $\mathbb Z^2$, so the coordinates of
$\operatorname{ab}_2(g)$ are coprime. This proves~(a). Conversely, a primitive
vector $v\in\mathbb Z^2$ is the first column of a matrix in
$\operatorname{GL}(2,\mathbb Z)$. Lift that matrix to an automorphism of
$F_2$; the image of $a$ is primitive and has abelianization $v$. This
proves~(b).

For~(c), let $g,h$ be primitive with the same abelianization vector $v$.
Extend them to ordered bases $(g,g_2)$ and $(h,h_2)$, and put
\[
 u=\operatorname{ab}_2(g_2),\qquad u'=\operatorname{ab}_2(h_2).
\]
After replacing $h_2$ by $h_2^{-1}$ and $u'$ by $-u'$, if necessary, assume
$\det(v,u)=\det(v,u')$. Then $\det(v,u-u')=0$, and primitivity of the integral
vector $v$ gives $u-u'=kv$ for some $k\in\mathbb Z$. Replace $h_2$ by
$h^kh_2$ and continue to denote the modified second basis element by $h_2$.
This power of an elementary Nielsen transformation gives
$\operatorname{ab}_2(h_2)=kv+u'=u$.

Let $\alpha,\beta\in\Aut(F_2)$ carry $(a,b)$ to $(g,g_2)$ and $(h,h_2)$,
respectively. Their induced matrices on abelianization are equal, so
$\beta\alpha^{-1}\in\ker(\Aut(F_2)\to\operatorname{GL}(2,\mathbb Z))
=\operatorname{Inn}(F_2)$. Hence, for some $c\in F_2$,
$h=\beta(a)=(\beta\alpha^{-1})(g)=cgc^{-1}$. This proves~(c) and the stated
bijection.
\end{proof}

We will use the following immediate consequence.

\begin{cor}\label{c:ab10}
If $g\in F_2$ is primitive and
\[
 \operatorname{ab}_2(g)=(1,0),
\]
then $g$ is conjugate to $a$.
\end{cor}

\begin{proof}
Both $g$ and $a$ are primitive and have the same exponent-sum pair.  Apply
Proposition~\ref{p:oz-classification}(c).
\end{proof}

The automorphisms used below are the two sign changes
\[
 \iota_a:a\mapsto a^{-1},\ b\mapsto b,
 \qquad
 \iota_b:a\mapsto a,\ b\mapsto b^{-1},
\]
the interchange $a\leftrightarrow b$, and the shears
\[
 \tau_m:a\mapsto ab^{-m},\ b\mapsto b
 \qquad(m\in\mathbb Z).
\]
The sign changes and interchange are elementary Nielsen automorphisms, and
$\tau_m=\tau_1^m$. If $\operatorname{ab}_2(g)=(p,q)$, then
\[
 \operatorname{ab}_2(\tau_m(g))=(p,q-mp).
\]
A compact description of a sequence of these maps records each sign change
or interchange by a fixed symbol and each shear by its type and a signed
binary parameter $m$. A large $m$ is not expanded into $|m|$ individual
Nielsen moves.

\begin{lem}[normalizing a primitive abelianization vector]\label{l:abelian-normalization}
If $g\in F_2$ and $\operatorname{ab}_2(g)=(p,q)$ has
$\gcd(|p|,|q|)=1$, there exists $\Phi\in\Aut(F_2)$ such that
$\operatorname{ab}_2(\Phi(g))=(1,0)$.
\end{lem}

\begin{proof}
The vector $(p,q)$ extends to an integral basis of $\mathbb Z^2$. Thus
some matrix in $\operatorname{GL}(2,\mathbb Z)$ sends it to $(1,0)$.
Lift this matrix through the surjective abelianization homomorphism
$\Aut(F_2)\to\operatorname{GL}(2,\mathbb Z)$ recalled in the proof of
Proposition~\ref{p:oz-classification}.
\end{proof}

This statement requires only that the abelianization vector be primitive,
not that $g$ itself be primitive. Lemma~\ref{l:euclidean-automorphism}
will supply the additional algorithmic content: a polynomial-size
compressed construction of both the automorphism and its image of the
input word.

\subsection{Euclidean normalization and shortening}\label{ss:rank-two-euclidean}
After sign changes and, when necessary, an interchange, Euclidean division
$Q=mP+R$, with $0<P\le Q$ and $0\le R<P$, is realized on exponent sums
by $\tau_m$. Repeating these divisions normalizes a coprime pair to
$(1,0)$. This normalization need not shorten an arbitrary group element.
The next calculation explains the term \emph{Euclidean shortening} on
positive primitive representatives: one shear removes a common power of
$b$ from every $b$-block of the Christoffel word. This literal shortening
interpretation is not needed for the final correctness test in
Section~\ref{ss:rank-two-test}.

\begin{prop}[one Christoffel--Euclidean step]
\label{p:christoffel-step}
For integers $P\ge1$, $Q\ge0$, put
\[
 d_i(P,Q)=\left\lfloor\frac{iQ}{P}\right\rfloor
     -\left\lfloor\frac{(i-1)Q}{P}\right\rfloor,
 \qquad 1\le i\le P,
\]
and $C_{P,Q}=\prod_{i=1}^{P}ab^{d_i(P,Q)}$. Assume
$0<P\le Q$ and $\gcd(P,Q)=1$, write
\[
 Q=mP+R,\qquad m\ge1,\qquad 0\le R<P,
\]
and define $\tau_m\in\Aut(F_2)$ by
$\tau_m(a)=ab^{-m}$ and $\tau_m(b)=b$. Then:
\begin{enumerate}
\item $C_{P,Q}$ is cyclically reduced and primitive, has exponent-sum pair
$(P,Q)$, and represents the unique primitive conjugacy class with that pair;
\item $\tau_m(C_{P,Q})=C_{P,R}$;
\item The cyclic lengths satisfy
\[
 \lVert C_{P,R}\rVert_{\rm cyc}=P+R
 <\frac23(P+Q)
 =\frac23\lVert C_{P,Q}\rVert_{\rm cyc}.
\]
\end{enumerate}
\end{prop}

\begin{proof}
Part (a) is the Christoffel-word form of
Proposition~\ref{p:oz-classification}; see
\cite{OsborneZieschang,KasselReutenauer}. For (b),
\[
 d_i(P,Q)=m+
 \left(\left\lfloor\frac{iR}{P}\right\rfloor
 -\left\lfloor\frac{(i-1)R}{P}\right\rfloor\right).
\]
Thus $d_i(P,Q)-m=d_i(P,R)$. The terminal word obtained by applying the
substitution defining $\tau_m$ freely reduces as follows:
\[
 \red\left(\prod_{i=1}^{P}ab^{-m}b^{d_i(P,Q)}\right)
 =\prod_{i=1}^{P}ab^{d_i(P,Q)-m}
 =C_{P,R}.
\]
In particular, $\tau_m(C_{P,Q})=C_{P,R}$ in $F_2$.
For (c), the positive words have cyclic lengths $P+Q$ and $P+R$, and
\[
 2(P+Q)-3(P+R)=(2m-1)P-R>0,
\]
because $m\ge1$ and $R<P$.
\end{proof}

Thus, after possibly inverting or interchanging the basis elements, one
Euclidean division step deletes the common power $b^m$ from every
$b$-block of the Christoffel representative.  The remaining exponents are
$0$ or $1$, and after interchanging $a$ and $b$ the same operation can be
repeated.  In particular, the number of batched Euclidean stages is
$O(\log(P+Q))$.

\subsection{Compressed realization of Euclidean transformations}\label{ss:rank-two-compression}
We now implement the preceding Euclidean Nielsen transformations directly
on $\SLP$s. Unlike the fixed-family layered construction recalled in
Section~\ref{ss:descent}, the construction here permits shears with binary
parameters supplied by the Euclidean algorithm. The next lemma stores each
terminal image once even when that image is itself compressed.

\begin{lem}[compressed substitution]\label{l:compressed-substitution}
Let $\CG$ be an $\SLP$ over a fixed alphabet $\Sigma$, and suppose that an
$\SLP$ $\CH_z$ over $\Sigma$ is given for every $z\in\Sigma$. One can
construct in deterministic polynomial time an $\SLP$ $\CB$ producing the
image of $w_{\CG}$ under the homomorphism $z\mapsto w_{\CH_z}$ and satisfying
\[
 \lVert\CB\rVert\le\lVert\CG\rVert+
 \sum_{z\in\Sigma}\lVert\CH_z\rVert=S.
\]
Its encoded length is $O(S\log S)$ for $S\ge2$. If all input $\SLP$s are
in Chomsky normal form and each $w_{\CH_z}$ is nonempty, then $\CB$ may
also be chosen in Chomsky normal form with the same additive bound.
\end{lem}

\begin{proof}
Rename nonterminals so that all input programs have disjoint variable sets.
First include each image program $\CH_z$ once. Then copy the productions
of $\CG$ in dependency order, replacing each terminal $z$ by the root of
$\CH_z$ and each nonterminal by its copied counterpart. The copied root
of $\CG$ is the final nonterminal. Induction on the productions gives the
required output. Every right-hand-side symbol of $\CG$ is replaced by
exactly one symbol, so this construction has size exactly
$\lVert\CG\rVert+\sum_z\lVert\CH_z\rVert$. No normalization is needed.

For the normal-form assertion, write $A_*$ for the root of $\CG$. If its
production is $A_*\to\eps$, return the one-rule program producing $\eps$;
if its production is $A_*\to z$,
return a copy of $\CH_z$. Otherwise the root production is binary. Include
each $\CH_z$ once and process the productions of $\CG$ in dependency order,
associating a root $\theta(A)$ to every nonterminal $A$. For a terminal
production $A\to z$, let $\theta(A)$ be the root of that same copy of
$\CH_z$, without adding a unit production. For $A\to BC$, add the production
\[
 \theta(A)\longrightarrow\theta(B)\theta(C).
\]
The final nonterminal is $\theta(A_*)$. If $\CG$ has $t$ binary
productions, this construction has size $\sum_z\lVert\CH_z\rVert+3t\le S$:
each new binary production contributes one nonterminal and two
right-hand-side symbols. All productions are in Chomsky normal form,
since the image words are nonempty.

Both constructions process only the input productions and run in polynomial
time. The encoded-length bound follows from
Proposition~\ref{p:slp-size-comparison}.
\end{proof}

\begin{lem}[compressed powers]\label{l:compressed-powers}
Given $m\ge0$ in binary, one can construct $\SLP$s for $a^{\pm m}$ and
$b^{\pm m}$ of size
\[
 O\bigl(\blen(m)\bigr)
\]
and encoded length
\[
 O\bigl(\blen(m)\log(\blen(m)+2)\bigr),
\]
in time polynomial in the bit-size of $m$.
\end{lem}

\begin{proof}
For $m=0$, use the one-rule program whose root produces $\eps$. For
$m>0$, construct the powers with exponents $1,2,4,\dots$ by repeated squaring
and concatenate those corresponding to the nonzero binary digits of $m$.
This uses $O(\blen(m))$ nonterminals and right-hand-side symbols, proving the
structural-size bound. Proposition~\ref{p:slp-size-comparison} gives the
stated encoded-length bound. The inverse powers are obtained by
Lemma~\ref{l:slp-ops}(c), or by carrying out the same construction with the
inverse terminal.
\end{proof}

\begin{lem}[compressed Euclidean automorphism]\label{l:euclidean-automorphism}
Let $\CA$ be an $\SLP$ over $\{a,b\}^{\pm1}$ of size $n$, and put
\[
 g=g_{\CA},\qquad \operatorname{ab}_2(g)=(p,q).
\]
If $\gcd(|p|,|q|)=1$, then one can construct, in deterministic polynomial
time, an $\SLP$ $\CG$ and a description of an automorphism $\Phi\in\Aut(F_2)$
such that
\[
 \val(\CG)\text{ represents }\Phi(g)
 \quad\text{and}\quad
 \operatorname{ab}_2(\Phi(g))=(1,0).
\]
The automorphism is recorded as a sequence of the sign changes,
interchanges, and binary-parameter shears specified in
Section~\ref{ss:rank-two-algebra}.  This record has bit-size polynomial in
$n$, while $\CG$ has structural size, and hence encoded length, polynomial in
$n$.
\end{lem}

\begin{proof}
The exponent sums are computed by Lemma~\ref{l:slp-ops}(d). Since
$|p|+|q|\le |w_{\CA}|\le2^{O(n)}$, these integers have $O(n)$ bits. Normalize
$\CA$ once at the outset. By the constant-factor normalization in
Section~\ref{s:slp}, the normalized program has size $O(n)$ and encoded
length $O(n\log(n+2))$.

First apply, when necessary, the automorphisms
\[
 a\mapsto a^{-1},\ b\mapsto b,
 \qquad\text{and}\qquad
 a\mapsto a,\ b\mapsto b^{-1},
\]
so that the current exponent-sum pair is $(P,Q)=(|p|,|q|)$.  If $P=0$,
then $Q=1$, and we interchange $a$ and $b$.  Hence we may assume $P>0$.

While $Q\ne0$, proceed as follows.  If $P>Q$, apply the interchange
automorphism $a\leftrightarrow b$ and interchange $P,Q$.  We then have
$0<P\le Q$.  Compute
\[
 m=\left\lfloor\frac QP\right\rfloor,
 \qquad R=Q-mP,
\]
and apply
\[
 \tau_m(a)=ab^{-m},\qquad \tau_m(b)=b.
\]
On abelianization this sends $(P,Q)$ to
$(P,Q-mP)=(P,R)$. Continue with this new pair. Since the gcd is unchanged, the
process terminates at $(1,0)$.

Each sign change or interchange has constant-size terminal images.  To
apply $\tau_m$, use
\[
 a\mapsto ab^{-m},\qquad
 a^{-1}\mapsto b^m a^{-1},\qquad
 b\mapsto b,
 \qquad b^{-1}\mapsto b^{-1}.
\]
By Lemmas~\ref{l:compressed-substitution} and~\ref{l:compressed-powers},
the substitution is performed in polynomial time.
Choose the image programs in Chomsky normal form. Every sign change,
interchange, and shear used here has nonempty terminal images, so the
normal-form assertion of Lemma~\ref{l:compressed-substitution} applies at
every stage. Its additive bound shows that a shear with parameter $m$
increases the structural size by at most the total size of its image
$\SLP$s, which is $O(\log(m+1))$; a sign change or interchange adds only
constant size. No further normalization is needed between stages.

Let $S=P+Q$ immediately before a division stage, after any interchange.
The new sum is $P+R$. Since $Q=mP+R$, with $m\ge1$ and $0\le R<P$,
\[
 2(P+Q)-3(P+R)=(2m-1)P-R>0,
\]
and therefore
\[
 P+R<\frac23(P+Q)=\frac23S.
\]
Consequently there are $O(\log(|p|+|q|+1))=O(n)$ division stages.  Every
quotient $m$ is at most $|p|+|q|$, so
\[
 \sum_{\text{division stages}}\log(m+1)=O(n^2).
\]
The normalized initial program has size $O(n)$, and the total increase
from sign changes, interchanges, and shears is at most
\[
 O\!\left(1+\sum_{\text{division stages}}\log(m+1)\right)=O(n^2).
\]
The constant term covers the initial operations even when no division is
needed.
Thus every intermediate program has polynomial size.  Its encoded length is
polynomial in $n$ by Proposition~\ref{p:slp-size-comparison}.  The sequence
describing $\Phi$ has $O(n)$ stages and $O(n)$-bit parameters, hence
polynomial bit-size as well. The integer arithmetic
uses $O(n)$-bit integers, and there are only $O(n)$ substitutions, each on a
polynomial-size program. The total construction time is therefore polynomial.
\end{proof}

\subsection{The decision algorithm and its complexity}\label{ss:rank-two-test}
We can now prove the rank-two complexity result
(cf. Theorem~\ref{t:rank-two-intro}):

\begin{thm}\label{t:rank-two-main}
The compressed primitivity problem in $F_2$ is decidable in deterministic
polynomial time.  Equivalently,
\[
 \CPrim_2\in\mathsf P.
\]
\end{thm}

\begin{proof}
Reject invalid $\SLP$ encodings. For a valid input $\CA$, let
$g=g_{\CA}$ and compute
\[
 \operatorname{ab}_2(g)=(p,q)
\]
as in the proof of Lemma~\ref{l:euclidean-automorphism}.  If
$(p,q)=(0,0)$, reject.  Otherwise compute $\gcd(|p|,|q|)$ and reject if it is
not equal to $1$.  This is correct for every primitive input by
Proposition~\ref{p:oz-classification}(a).

Suppose now that $\gcd(|p|,|q|)=1$.  Apply
Lemma~\ref{l:euclidean-automorphism} to construct a polynomial-size $\SLP$
$\CG$ representing $\Phi(g)$, where
\[
 \operatorname{ab}_2(\Phi(g))=(1,0).
\]
Use Proposition~\ref{p:cyclic-reduction}(b) to construct an $\SLP$ for a
freely and cyclically reduced conjugate $u$ of $\Phi(g)$, and compute
$|u|$ in binary.  Accept if and only if
\[
 |u|=1.
\]
All these operations take deterministic polynomial time.

For soundness, suppose that the algorithm accepts.  A cyclically reduced
word of length one is one of $a^{\pm1},b^{\pm1}$.  Conjugation does not
change exponent sums, and $\operatorname{ab}_2(u)=(1,0)$; hence $u=a$.
Therefore $\Phi(g)$ is conjugate to $a$, so $\Phi(g)$ is primitive.  Since
$\Phi$ is an automorphism, $g$ is primitive.

For completeness, suppose that $g$ is primitive.  Then $\Phi(g)$ is
primitive and has exponent-sum pair $(1,0)$.  By
Corollary~\ref{c:ab10}, $\Phi(g)$ is conjugate to $a$.  Its cyclically
reduced length is therefore one, and the algorithm accepts.
\end{proof}

\begin{rem}[quantitative complexity in rank two]
\label{r:rank-two-complexity}
For $N\ge1$, let $T_{\rm cyc}(N)$ denote the worst-case number of bit
operations used by a fixed deterministic implementation of
Proposition~\ref{p:cyclic-reduction}(b), on inputs $\mathcal D$ with
$\beta(\mathcal D)\le N$, to
construct an $\SLP$ for the cyclically reduced form and to compute its length.
Suppose that
\[
 T_{\rm cyc}(N)=O(N^d)
\]
for some constant $d\ge1$.

Let $n=\lVert\CA\rVert$. If no division is needed in the Euclidean
construction of Lemma~\ref{l:euclidean-automorphism}, only a bounded number
of sign changes and interchanges occur, and the bounds below follow
directly. Otherwise, let $m_1,\dots,m_k$, with $k\ge1$, be the nonzero
Euclidean quotients, and let $S_i$ be the sum of the two nonnegative
coordinates immediately before the $i$th division, for $1\le i\le k$.
Put $S_{k+1}=1$, the sum of the final coordinates. If
\[
 Q_i=m_iP_i+R_i,\qquad 0\le R_i<P_i,
\]
then, up to an interchange of coordinates,
\[
 S_i=P_i+Q_i,\qquad S_{i+1}=P_i+R_i.
\]
Moreover,
\[
 2S_i-(m_i+1)S_{i+1}
 =(m_i+1)P_i-(m_i-1)R_i>0,
\]
and hence
\[
 m_i+1<2\frac{S_i}{S_{i+1}}.
\]
Taking logarithms and summing gives
\[
 \sum_{i=1}^k\log_2(m_i+1)
 < k+\log_2\frac{S_1}{S_{k+1}}
 =O(n),
\]
because $k=O(n)$, $S_1\le2^{O(n)}$, and the final pair is $(1,0)$, so
$S_{k+1}=1$.  Consequently,
\[
 \sum_{i=1}^k\blen(m_i)=O(n).
\]

It follows that the record of the automorphism $\Phi$ has bit-size $O(n)$ and
that the program $\CG$ representing $\Phi(g_{\CA})$ may be chosen with
\[
 \lVert\CG\rVert=O(n),
 \qquad
 \beta(\CG)=O\bigl(n\log(n+2)\bigr),
\]
where the second estimate follows from
Proposition~\ref{p:slp-size-comparison}.
The sequential substitution construction above uses
$O(n^2\log(n+2))$ grammar bit operations: there are $O(n)$ stages, and each
stage rewrites an $O(n)$-size grammar whose indices have
$O(\log(n+2))$ bits. The Euclidean algorithm performs $O(n)$ divisions and
associated operations on $O(n)$-bit integers. With conservative schoolbook
arithmetic, these use $O(n^3)$ bit operations, so all steps preceding the
final cyclic reduction have the same bound. Consequently, the algorithm in
Theorem~\ref{t:rank-two-main} runs in
\[
 O\!\left(
   n^3+T_{\rm cyc}\bigl(Cn\log(n+2)\bigr)
 \right)
\]
bit operations for some constant $C$.  In particular, if
$T_{\rm cyc}(N)=O(N^d)$, then the running time is
\[
 O\!\left(n^3+n^d\log^d(n+2)\right).
\]

Schleimer's compressed cyclic-reduction result supplies some absolute
polynomial degree $d$ for which the displayed conditional estimate applies.
The published proof does not state or optimize the numerical value of $d$,
and its precise value for a concrete implementation remains to be computed.
\end{rem}

\begin{rem}[literal shortening]\label{r:literal-shortening}
For a positive primitive conjugacy class with exponent-sum pair $(P,Q)$,
suppose, after interchanging $a$ and $b$ if necessary, that $0<P\le Q$,
and put $m=\lfloor Q/P\rfloor$. The cases with one coordinate zero are
already represented by a single basis letter.
Proposition~\ref{p:christoffel-step} shows that the shear $\tau_m$ performs
simultaneous cancellation of the common power $b^m$ in the Christoffel
representative.  The $b$-block exponents are $m$ or $m+1$, but
their cyclic order is the additional arithmetic information supplied by
the Christoffel formula.  After the shear, the residual exponents are $0$
or $1$, and the roles of $a$ and $b$ may be interchanged.  Thus the
Euclidean process shortens the cyclic length of a primitive conjugacy
class by a uniform multiplicative factor at each division stage.  The algorithm above does not need to recover
the exponentially long block decomposition: it performs the same shears
directly on the input $\SLP$ and makes a single cyclic-length test at the
end.
\end{rem}

\section*{Disclosure of AI use}

ChatGPT was used during the preparation of this manuscript. The author
independently checked and edited all mathematical arguments and takes full
responsibility for the final content.

\end{document}